\documentclass[12pt]{amsart}
\usepackage[utf8]{inputenc}
\usepackage[english]{babel}
\usepackage{amsmath,amssymb,a4wide,latexsym,amscd,amsthm}
\usepackage{graphics}

\theoremstyle{plain}
\newtheorem{theorem}{Theorem}[section]
\newtheorem{lemma}{Lemma}[section]
\newtheorem{proposition}{Proposition}[section]

\theoremstyle{definition}

\newtheorem{remark}{Remark}[section]

\numberwithin{equation}{section}

\begin{document}

\title[]
{On the inverse best approximation property of systems of subspaces
of a Hilbert space}

\author[]{Ivan Feshchenko}
\address{Institute of Mathematics of NAS of Ukraine, Kyiv, Ukraine}
\email{ivanmath007@gmail.com}

\begin{abstract}
Let $H$ be a Hilbert space and $H_1,...,H_n$ be closed subspaces of $H$.
Denote by $P_k$ the orthogonal projection onto $H_k$, $k=1,2,...,n$.
Following Patrick L. Combettes and Noli N. Reyes,
we will say that the system of subspaces $H_1,...,H_n$ possesses the
inverse best approximation property (IBAP) if for arbitrary elements
$x_1\in H_1,...,x_n\in H_n$ there exists an element $x\in H$ such that
$P_k x=x_k$ for all $k=1,2,...,n$.
We provide various new necessary and sufficient conditions
for a system of $n$ subspaces to possess the IBAP.
Using the main characterization theorem,
we study properties of the systems of subspaces which possess the IBAP,
obtain a sufficient condition for a system of subspaces to possess the IBAP,
and provide examples of systems of subspaces which possess the IBAP.

These results are applied to a problem of probability theory.
Let $(\Omega,\mathcal{F},\mu)$ be a probability space and 
$\mathcal{F}_1,...,\mathcal{F}_n$ be sub-$\sigma$-algebras of $\mathcal{F}$.
We will say that the collection $\mathcal{F}_1,...,\mathcal{F}_n$ possesses
the inverse marginal property (IMP) if for arbitrary random variables
$\xi_1,...,\xi_n$ such that 
(1) $\xi_k$ is $\mathcal{F}_k$-measurable, $k=1,2,...,n$;
(2) $E|\xi_k|^2<\infty$, $k=1,2,...,n$;
(3) $E\xi_1=E\xi_2=...=E\xi_n$,
there exists a random variable $\xi$ such that $E|\xi|^2<\infty$ and 
$E(\xi|\mathcal{F}_k)=\xi_k$ for all $k=1,2,...,n$.
We will show that a collection of sub-$\sigma$-algebras possesses the IMP
if and only if the system of corresponding marginal subspaces possesses the IBAP.
We consider two examples; in the first example $\Omega=\mathbb{N}$,
in the second example $\Omega=[a,b)$.
For these examples we establish relations between 
the IMP, the IBAP, closedness of the sum of marginal subspaces and
``fast decreasing'' of tails of the measure $\mu$.
Also, we provide a sufficient condition for a collection
of sub-$\sigma$-algebras to possess the IMP.
\end{abstract}



\maketitle

\section{Introduction}

\subsection{}
Let $H$ be a real or complex Hilbert space and $H_1,...,H_n$ be closed
subspaces of $H$.
Denote by $P_k$ the orthogonal projection onto $H_k$, $k=1,2,...,n$.
Following Patrick L. Combettes and Noli N. Reyes \cite{CR10},
we will say that the system of subspaces $H_1,...,H_n$ possesses the
inverse best approximation property (IBAP) if for arbitrary elements
$x_1\in H_1,...,x_n\in H_n$ there exists an element $x\in H$ such that
$P_k x=x_k$ for all $k=1,2,...,n$.

The simplest example of a system of subspaces which possesses the IBAP
is a system of pairwise orthogonal subspaces.
In this case a needed element $x$ can be defined by $x:=x_1+...+x_n$.

\subsection{}
Two remarks are in order.

\begin{remark}
The IBAP depends only on the geometry of relative position of the subspaces.
More precisely, the following result is valid.

\begin{proposition}\label{prop:relative position}
Let $H$ be a Hilbert space and $H_1,...,H_n$ be closed subspaces of $H$.
The system of subspaces $H_1,...,H_n$ possesses the IBAP in $H$
if and only if the system possesses the IBAP in
$\overline{H_1+...+H_n}$.
\end{proposition}

Note that $\overline{H_1+...+H_n}$ is the smallest closed subspace
that contains $H_1,...,H_n$.
\end{remark}

\begin{remark} 
The question about existence of a solution to more
general systems of equations $T_k x=v_k$, $k=1,2,...,n$ can be reduced to the
question about the IBAP of the system of subspaces 
$(\ker(T_1))^\bot,...,(\ker(T_n))^\bot$.
More precisely, the following result is valid.

\begin{proposition}\label{prop:basic}
Let $H$ be a Hilbert space, $V_1,...,V_n$ be vector spaces, and
$T_1:H\to V_1,...,T_n:H\to V_n$ be linear operators.
Consider the following statements:

(1) for arbitrary elements $v_1\in Ran(T_1),...,v_n\in Ran(T_n)$ there exists
an element $x\in H$ such that $T_k x=v_k$, $k=1,2,...,n$;

(2) the system of subspaces $(\ker(T_1))^\bot,...,(\ker(T_n))^\bot$ possesses
the IBAP.

Then $(1)\Rightarrow (2)$.
Moreover, if each subspace $\ker(T_k)$, $k=1,2,...,n$ is closed, then
$(2)\Rightarrow (1)$.
\end{proposition}
\end{remark}

\subsection{}
The IBAP is closely related to the notion of linear independence.
Recall that a system of subspaces $H_1,...,H_n$ is said to be linearly independent
if an equality $x_1+...+x_n=0$, where $x_1\in H_1,...,x_n\in H_n$, implies that
$x_1=...=x_n=0$.
It is easily seen \cite[Proposition 2.2]{CR10} 
that if elements $x_1\in H_1,...,x_n\in H_n$ are
such that $x_1+...+x_n=0$ and $(x_1,...,x_n)\neq (0,...,0)$, then
there is no $x\in H$ such that $P_k x=x_k$ for all $k=1,2,...,n$.
Conclusion \cite[Corollary 2.3]{CR10}: if a system of subspaces possesses the IBAP, 
then this system is linearly independent.
However, even for $n=2$ the linear independence of $H_1, H_2$,
i.e., the condition that $H_1\cap H_2=\{0\}$,
is not sufficient for a pair $H_1, H_2$ to possess the IBAP
(see \cite[Example 2.4]{CR10}).
It turns out that the linear independence is equivalent to the approximate IBAP.
More precisely, the following result \cite[Proposition 2.5]{CR10} is valid.

\begin{proposition}
Let $H$ be a Hilbert space and $H_1,...,H_n$ be closed subspaces of $H$.
The following statements are equivalent:

(1) the system of subspaces $H_1,...,H_n$ is linearly independent;

(2) for arbitrary elements $x_1\in H_1,...,x_n\in H_n$ and every $\varepsilon>0$
there exists an element $x\in H$ such that $\|P_k x-x_k\|\leqslant\varepsilon$
for all $k=1,2,...,n$.
\end{proposition}

\subsection{}
In \cite[Theorem 2.8]{CR10} P.L. Combettes and N.N. Reyes obtained
various necessary and sufficient conditions for a system of subspaces to possess
the IBAP.
In particular, for $n=2$ the following statements are equivalent
(for a more complete picture see \cite[Corollary 2.12]{CR10}):

(1) a pair of subspaces $H_1,H_2$ possesses the IBAP;

(2) $H_1\cap H_2=\{0\}$ and $H_1+H_2$ is closed in $H$;

(3) $H_1^\bot+H_2^\bot=H$.

In Section~\ref{S:mct} we provide various new necessary and sufficient conditions
for a system of $n$ subspaces to possess the IBAP (see Theorem~\ref{T:main}).
For example, we will show that the following conditions are equivalent:

(1) a system of subspaces $H_1,...,H_n$ possesses the IBAP;

(2) the subspaces $H_1,...,H_n$ are linearly independent
and their sum $H_1+...+H_n$ is closed in $H$;

(3) $\sum_{i=1}^n\bigcap_{j\neq i}H_j^\bot=H$.

In light of these results, it is worth mentioning that V.S. Sunder
in \cite{S88} study the structure of $n$-tuples of closed subspaces
$H_1,...,H_n$ of a Hilbert space $H$ such that 
$H_1,...,H_n$ are linearly independent and $H_1+...+H_n=H$.
The results of V.S. Sunder can be transferred to the systems of subspaces
which possess the IBAP (a system of subspaces $H_1,...,H_n$
can be considered as a system of subspaces of the Hilbert space $H_1+...+H_n$).

In Section~\ref{S:further results}, using Theorem~\ref{T:main},
we study properties of the systems of subspaces which possess the IBAP,
obtain a sufficient condition for a system of subspaces to possess the IBAP,
and provide examples of systems of subspaces which possess the IBAP 
(see Theorems~\ref{T:IBAP and Riesz}, \ref{T:eigenspaces}, \ref{T:root subspaces},
\ref{T:stability}).

In \cite[Section 4]{CR10} P.L. Combettes and N.N. Reyes
applied their results to problems of harmonic analysis,
integral equations, signal theory, and wavelet frames.
In Section~\ref{S:IMP} we apply results of Sections~\ref{S:mct}
and \ref{S:further results} to a problem of probability theory.
Let $(\Omega,\mathcal{F},\mu)$ be a probability space and 
$\mathcal{F}_1,...,\mathcal{F}_n$ be sub-$\sigma$-algebras of $\mathcal{F}$.
We will say that the collection $\mathcal{F}_1,...,\mathcal{F}_n$ possesses
the inverse marginal property (IMP) if for arbitrary random variables
$\xi_1,...,\xi_n$ such that 

(1) $\xi_k$ is $\mathcal{F}_k$-measurable, $k=1,2,...,n$;

(2) $E|\xi_k|^2<\infty$, $k=1,2,...,n$;

(3) $E\xi_1=E\xi_2=...=E\xi_n$,

there exists a random variable $\xi$ such that $E|\xi|^2<\infty$ and 
$E(\xi|\mathcal{F}_k)=\xi_k$ for all $k=1,2,...,n$.

The simplest example of a collection of sub-$\sigma$-algebras which possesses the IMP 
is a system of pairwise independent sub-$\sigma$-algebras.
In this case a needed random variable $\xi$ 
can be defined by $\xi:=\xi_1+...+\xi_n-(n-1)a$, where
$a:=E\xi_1=E\xi_2=...=E\xi_n$.

We will show that a collection of sub-$\sigma$-algebras possesses the IMP
if and only if the system of corresponding marginal subspaces possesses the IBAP
(see Theorem~\ref{T:IMP}).
We consider two examples; in the first example $\Omega=\mathbb{N}$
(see Subsection~\ref{SS:Omega=N}),
in the second example $\Omega=[a,b)$
(see Subsection~\ref{SS:Omega=[a,b)}).
For these examples we establish relations between 
the IMP, the IBAP, closedness of the sum of marginal subspaces and
``fast decreasing'' of tails of the measure $\mu$
(see Theorems~\ref{T:N} and \ref{T:[a,b)}).
Also, we provide a sufficient condition for a collection
of sub-$\sigma$-algebras to possess the IMP
(see Theorem~\ref{T:sufficient for IMP}).

A few of our results, namely, equivalence $(1)\Leftrightarrow(2)$
in Theorem~\ref{T:main} and Theorem~\ref{T:sufficient for IMP}
are contained in \cite[Examples 4.7 and 4.8]{F12}.
Here we provide clearer proofs of the results.

\subsection{Proof of Proposition \ref{prop:relative position}}
Assume that the system $H_1,...,H_n$ possesses the IBAP in
$\overline{H_1+...+H_n}$.
Then, clearly, this system possesses the IBAP in $H$.

Conversely, assume that the system $H_1,...,H_n$ possesses the IBAP in $H$.
Let us show that the system possesses the IBAP in $\overline{H_1+...+H_n}$.
Denote by $Q$ the orthogonal projection onto $\overline{H_1+...+H_n}$.
Consider arbitrary elements $x_1\in H_1,...,x_n\in H_n$.
There exists an element $x\in H$ such that $P_k x=x_k$, $k=1,2,...,n$.
Consider the element $Qx\in\overline{H_1+...+H_n}$.
We have $P_k(Qx)=P_k x=x_k$, $k=1,2,...,n$.
This means that the system $H_1,...,H_n$ possesses the IBAP in $\overline{H_1+...+H_n}$.

\subsection{Proof of Proposition \ref{prop:basic}}
Denote by $P_k$ the orthogonal projection onto $(\ker(T_k))^\bot$, $k=1,2,...,n$.

Let us prove that $(1)\Rightarrow (2)$.
Consider arbitrary $x_1\in (\ker(T_1))^\bot,...,x_n\in (\ker(T_n))^\bot$.
Set $v_k:=T_k x_k$, $k=1,2,...,n$.
We know that there exists an element $x\in H$ such that $T_k x=v_k$, $k=1,2,...,n$.
Then $T_k x=T_k x_k$, $T_k(x-x_k)=0$, $x-x_k\in \ker(T_k)$, $k=1,2,...,n$.
It follows that $P_k(x-x_k)=0$ and $P_k x=x_k$ for $k=1,2,...,n$.
This means that the system of subspaces $(\ker(T_1))^\bot,...,(\ker(T_n))^\bot$ 
possesses the IBAP.

Now assume that $\ker(T_k)$ is closed for all $k=1,2,...,n$.
Let us prove that $(2)\Rightarrow (1)$.
Consider arbitrary $v_1\in Ran(T_1),...,v_n\in Ran(T_n)$.
Then we can write $v_k=T_k x_k$ for some $x_k\in H$, $k=1,2,...,n$.
Consider the elements $P_k x_k\in (\ker(T_k))^\bot$, $k=1,2,...,n$.
We know that there exists an element $x\in H$ such that $P_k x=P_k x_k$
for all $k=1,2,...,n$.
Then $P_k(x-x_k)=0$.
It follows that $x-x_k\in ((\ker(T_k))^\bot)^\bot=\overline{\ker(T_k)}=\ker(T_k)$.
Thus $T_k(x-x_k)=0$ and $T_k x=T_k x_k=v_k$ for all $k=1,2,...,n$.

\section{Notation and auxiliary notions}

\subsection{Notation}
Throughout this paper $H$ is a real or complex Hilbert space.
The inner product in $H$ is denoted by $\langle\cdot,\cdot\rangle$ and 
$\|\cdot\|$ stands for the corresponding norm, $\|x\|=\sqrt{\langle x,x\rangle}$.
By a subspace of $H$ we will mean a linear subset of $H$.
For Hilbert spaces $H_1,...,H_n$ we denote by 
$H_1\oplus H_2\oplus...\oplus H_n$ their orthogonal direct sum.
The kernel and range of an operator $T$ will be denoted by 
$\ker(T)$ and $Ran(T)$, respectively.
For a continuous linear operator $T$ between two Hilbert spaces
we denote by $T^*$ its adjoint.
All vectors are vector-columns; the superscript ``t'' means transpose.

\subsection{Auxiliary notions}
For convenience of the reader we provide a ``vocabulary''
of several notions used in the paper.

\textbf{Orthogonal complement.}
Let $H$ be a Hilbert space and $M$ be a nonempty subset of $H$.
Denote by $M^\bot$ the set of all elements $x\in H$ such that
$\langle x,y\rangle=0$ for every $y\in M$.
Note that $M^\bot$ is a closed subspace of $H$.
If $M$ is a closed subspace of $H$, then $H=M+M^\bot$.

\textbf{Sum of subspaces.}
Let $V$ be a vector space and $V_1,...,V_n$ be subspaces of $V$.
Define the sum of $V_1,...,V_n$ in the natural way, namely,
$$
V_1+...+V_n:=\{x_1+...+x_n|x_1\in V_1,...,x_n\in V_n\}.
$$
It is clear that $V_1+...+V_n$ is a subspace of $V$.

\textbf{Linear independence.}
Let $V$ be a vector space and $V_1,...,V_n$ be subspaces of $V$.
The system of subspaces $V_1,...,V_n$ is said to be linearly independent 
if an equality $x_1+...+x_n=0$, where $x_1\in V_1,...,x_n\in V_n$,
implies that $x_1=...=x_n=0$.

\textbf{Isomorphism.}
Let $H$ and $K$ be two Hilbert spaces.
A mapping $A:H\to K$ is called an isomorphism if
(1) $A$ is a bijection;
(2) $A$ is linear;
(3) the operators $A:H\to K$ and $A^{-1}:K\to H$ are continuous.
From the Banach inverse mapping theorem it follows that
a continuous linear operator $A:H\to K$ with
$\ker(A)=\{0\}$ and $Ran(A)=K$ is an isomorphism.

\textbf{Isomorphic embedding.}
Let $H$ and $K$ be two Hilbert spaces.
A continuous linear operator $A:H\to K$ is called 
an isomorphic embedding if $A$, 
considered as the operator from $H$ to $Ran(A)$, is an isomorphism.
For a continuous linear operator $A:H\to K$ the following statements
are pairwise equivalent:

(1) $A$ is an isomorphic embedding;

(2) $\ker(A)=\{0\}$ and $Ran(A)$ is closed in $K$;

(3) there exists a number $c>0$ such that $\|Ax\|\geqslant c\|x\|$
for all $x\in H$;

(4) the operator $A^* A:H\to H$ is an isomorphism. 

\textbf{Isomorphic systems of subspaces.}
Let $H$ be a Hilbert space and $H_1,...,H_n$ be closed subspaces of $H$.
Let $K$ be a Hilbert space and $K_1,...,K_n$ be closed subspaces of $K$.
Following \cite{EW06}, we will say that systems $(H;H_1,...,H_n)$ and
$(K;K_1,...,K_n)$ are isomorphic if there exists an isomorphism
$\varphi:H\to K$ such that $\varphi(H_i)=K_i$ for $i=1,2,...,n$.
As Masatoshi Enomoto and Yasuo Watatani write in \cite{EW06},
the fact that the systems are isomorphic means that the relative
positions of $n$ subspaces $(H_1,...,H_n)$ in $H$ and
$(K_1,...,K_n)$ in $K$ are same under disregarding angles.

\textbf{Projection.}
Let $V$ be a vector space and $M$ be a subspace of $V$.
A linear operator $P:V\to V$ is called a projection onto $M$ if
$Px\in M$ for all $x\in V$ and $Px=x$ for all $x\in M$.
A linear operator $P:V\to V$ is called a projection if
there exists a subspace $M$ such that $P$ is a projection onto $M$.
It is easily seen that a linear operator $P:V\to V$ is a projection
if and only if $P^2=P$.

\textbf{The Gram operator of a system of subspaces.}
Let $H$ be a Hilbert space and $H_1,...,H_n$ be closed subspaces of $H$.
Denote by $P_k$ the orthogonal projection onto $H_k$, $k=1,2,...,n$.
The operator 
$G=G(H_1,...,H_n):H_1\oplus...\oplus H_n\to H_1\oplus...\oplus H_n$
defined by its block decomposition $G=(G_{ij}|i,j=1,2,...,n)$,
where $G_{ij}:=P_i|_{H_j}:H_j\to H_i$, $i,j=1,2,...,n$
is called the Gram operator of the system $(H;H_1,...,H_n)$.
Properties of the Gram operator can be found in \cite[Section 2]{SF13}.

\textbf{Equivalent inner products.}
Let $V$ be a vector space and $\langle\cdot,\cdot\rangle_1$ and
$\langle\cdot,\cdot\rangle_2$ be two inner products in $V$.
The inner products are said to be equivalent if the corresponding
norms $\|\cdot\|_1$ and $\|\cdot\|_2$ are equivalent, i.e., if
there exist numbers $a>0$ and $b>0$ such that
$a\|x\|_1\leqslant\|x\|_2\leqslant b\|x\|_1$ for all $x\in V$.

\textbf{Inclination ($\delta$).}
Let $X$ be a normed vector space and $Y,Z$ be subspaces of $X$.
Assume that $Y\neq\{0\}$.
The inclination of $Y$ to $Z$, $\delta(Y,Z)$, is defined by
$$
\delta(Y,Z):=\inf\{dist(y,Z)\,|\,y\in Y, \|y\|=1\},
$$
where $dist(y,Z):=\inf\{\|y-z\|\,|\,z\in Z\}$.
It is clear that 
$$
\delta(Y,Z)=\inf\{\dfrac{dist(y,Z)}{\|y\|}\,|\,y\in Y\setminus\{0\}\}.
$$
If $Y=\{0\}$, then we set $\delta(Y,Z):=+\infty$.
The notion of inclination is very well known in the geometry of Banach spaces. 
Properties on the inclination can be found, 
for example, in \cite[Chapter 1]{GL05}.

\textbf{Opening ($\theta$).}
Let $H$ be a Hilbert space and $M,N$ be two closed subspaces of $H$.
Denote by $P,Q$ the orthogonal projections onto $M,N$, respectively.
The number $\theta(M,N):=\|P-Q\|$ is called the opening between $M$ and $N$.
It is easily seen that $\theta(M,N)\in[0,1]$ and $\theta$ is a metric
on the space of all closed subspaces of $H$.
For other properties of the opening and formulas for the opening
see, e.g., \cite{O94}.

\textbf{Riesz families.}
Let $H$ be a separable infinite dimensional Hilbert space.
A sequence $\{v_n|n\geqslant 1\}$ of elements of the space $H$ is called
a Riesz basis for $H$ if there exist an orthonormal basis $\{e_n|n\geqslant 1\}$
of $H$ and an isomorphism $A:H\to H$ such that $v_n=Ae_n$, $n\geqslant 1$.
A sequence $\{v_n|n\geqslant 1\}$ is called a Riesz sequence if
$\{v_n|n\geqslant 1\}$ is a Riesz basis for the closure of the subspace spanned by
the sequence.
The notions of Riesz basis and Riesz sequence are very well known
(see, e.g., \cite[Section 6.2]{GK69} and \cite[Section 7.2]{H11}).
Various necessary and sufficient conditions for a sequence to be a Riesz basis of $H$
can be found, for example, in \cite[Section 6.2]{GK69} and 
\cite[Section 7.2 and Theorem 8.32]{H11}.
Using these results one can get necessary and sufficient conditions for
a sequence to be a Riesz sequence.
For example, it is well known that a sequence $\{v_n|n\geqslant 1\}$
is a Riesz basis for $H$ if and only if the subspace spanned by
$\{v_n|n\geqslant 1\}$ is dense in $H$ and there exist numbers
$\varepsilon>0$ and $C>0$ such that
\begin{equation}\label{ineq:Riesz}
\varepsilon\left(\sum_{k=1}^N |a_k|^2\right)^{1/2}\leqslant
\|a_1 v_1+...+a_N v_N\|\leqslant
C\left(\sum_{k=1}^N |a_k|^2\right)^{1/2}
\end{equation}
for every $N\geqslant 1$ and arbitrary scalars $a_1,...,a_N$.
It follows that a sequence $\{v_n|n\geqslant 1\}$ is a Riesz sequence
if and only if there exist numbers $\varepsilon>0$ and $C>0$
such that inequalities~\eqref{ineq:Riesz} hold
for every $N\geqslant 1$ and arbitrary scalars $a_1,...,a_N$.

Now let $H$ be arbitrary Hilbert space.
A family $\{v_\alpha|\alpha\in M\}$ of elements of $H$ is called a Riesz basis
for $H$ if there exist an orthonormal basis $\{e_\alpha|\alpha\in M\}$
of $H$ and an isomorphism $A:H\to H$ such that $v_\alpha=Ae_\alpha$,
$\alpha\in M$.
A family $\{v_\alpha|\alpha\in M\}$ is called a Riesz family if
$\{v_\alpha|\alpha\in M\}$ is a Riesz basis for the closure of the subspace
spanned by the family.
One can check that a family $\{v_\alpha|\alpha\in M\}$ is a Riesz family
if and only if there exist numbers $\varepsilon>0$ and $C>0$
such that
$$
\varepsilon\left(\sum_{\alpha\in F}|a_\alpha|^2\right)^{1/2}\leqslant
\|\sum_{\alpha\in F}a_\alpha v_\alpha\|\leqslant
C\left(\sum_{\alpha\in F}|a_\alpha|^2\right)^{1/2}
$$
for every finite subset $F\subset M$ and arbitrary scalars $a_\alpha$, $\alpha\in F$.

\section{IBAP: the main characterization theorem}\label{S:mct}

\begin{theorem}\label{T:main}
Let $H$ be a Hilbert space and $H_1,...,H_n$ be closed subspaces of $H$.
The following statements are equivalent:

(1) the system of subspaces $H_1,...,H_n$ possesses the IBAP;

(2) the subspaces $H_1,...,H_n$ are linearly independent
and their sum $H_1+...+H_n$ is closed in $H$;

(3)
$\inf\{
\dfrac{\|x_1+...+x_n\|}{\sqrt{\|x_1\|^2+...+\|x_n\|^2}}\,|\,
x_1\in H_1,...,x_n\in H_n,
(x_1,...,x_n)\neq (0,...,0)\}>0$;

(4) the Gram operator $G=G(H_1,...,H_n)$ is an isomorphism;

(5) there exist a Hilbert space $K$ and pairwise orthogonal closed
subspaces $K_1,...,K_n$ of the space $K$ such that the system
$(H;H_1,...,H_n)$ is isomorphic to $(K;K_1,...,K_n)$;

(6) there exists an inner product $\langle\cdot,\cdot\rangle_0$ in $H$
which is equivalent to the original inner product in $H$ and such that
the subspaces $H_1,...,H_n$ are pairwise orthogonal in 
$(H;\langle\cdot,\cdot\rangle_0)$.

(7) there exist continuous linear projections $E_k:H\to H$, $k=1,...,n$,
such that $E_i E_j=0$ for every pair $i\neq j$ and $Ran(E_k)=H_k$ for
every $k=1,...,n$;

(8) $\delta(H_i,\sum_{j\neq i}H_j)>0$ for every $i=1,...,n$;

(9) $\sum_{i=1}^n\bigcap_{j\neq i}H_j^\bot=H$;

(10) $H_i^\bot+\bigcap_{j\neq i}H_j^\bot=H$ for arbitrary $i=1,2,...,n$.
\end{theorem}
\begin{proof}
Denote by $P_k$ the orthogonal projection onto $H_k$, $k=1,...,n$.
Let us introduce two auxiliary operators $J$ and $S$.
Define the operator $J:H\to H_1\oplus...\oplus H_n$ by
$$
Jx:=(P_1 x,...,P_n x)^t,\qquad x\in H.
$$
Define the operator $S:H_1\oplus...\oplus H_n\to H$ by
$$
S(x_1,...,x_n)^t:=x_1+...+x_n,\qquad (x_1,...,x_n)^t\in H_1\oplus...\oplus H_n.
$$
It is easily seen that $J^*=S$.
It is clear that the system $H_1,...,H_n$ possesses the IBAP
$\Leftrightarrow$ the operator $J$ is surjective
$\Leftrightarrow$ the operator $J^*=S$ is an isomorphic embedding.

$(1)\Leftrightarrow (2)$
$S$ is an isomorphic embedding
$\Leftrightarrow$ $\ker(S)=\{0\}$ and $Ran(S)$ is closed in $H$
$\Leftrightarrow$ the subspaces $H_1,...,H_n$ are linearly independent
and their sum $H_1+...+H_n$ is closed in $H$.
Thus $(1)\Leftrightarrow (2)$.

$(1)\Leftrightarrow (3)$
$S$ is an isomorphic embedding
$\Leftrightarrow$ there exists a number $c>0$ such that
$\|S(x_1,...,x_n)^t\|\geqslant c\|(x_1,...,x_n)^t\|$ for every element
$(x_1,...,x_n)^t$ of the space $H_1\oplus...\oplus H_n$.
The last inequality can be rewritten as
$$
\|x_1+...+x_n\|\geqslant c\sqrt{\|x_1\|^2+...+\|x_n\|^2}.
$$
Thus $(1)\Leftrightarrow (3)$.

$(1)\Leftrightarrow (4)$
$S$ is an isomorphic embedding
$\Leftrightarrow$ the operator $S^* S$ is an isomorphism.
We have $S^* S=JS=G$.
Thus $(1)\Leftrightarrow (4)$.

$(2)\Rightarrow(5)$
This implication follows from \cite[Proposition 12.12]{EW06}.
For convenience of the reader we present the proof.
Assume that the subspaces $H_1,...,H_n$ are linearly independent
and their sum $H_1+...+H_n$ is closed in $H$.
Define the closed subspace $N:=H\ominus (H_1+...+H_n)$.
Define the operator $\varphi:H_1\oplus...\oplus H_n\oplus N\to H$ by
$$
\varphi(x_1,...,x_n,x)^t:=x_1+...+x_n+x, \qquad 
(x_1,...,x_n,x)^t\in H_1\oplus...\oplus H_n\oplus N.
$$
Since $H_1,...,H_n$ are linearly independent, we conclude that
$H_1,...,H_n,N$ are also linearly independent.
It follows that $\ker(\varphi)=\{0\}$.
Further, $Ran(\varphi)=H_1+...+H_n+N=H$.
Thus $\varphi$ is an isomorphism.
For every $i=1,...,n$ define the subspace 
$\widetilde{H}_i\subset H_1\oplus...\oplus H_n\oplus N$ by
$$
\widetilde{H}_i:=
\{0\}\oplus...\oplus\{0\}\oplus H_i\oplus\{0\}\oplus...\oplus\{0\}=
\{(0,...,0,x_i,0,...,0)|x_i\in H_i\}.
$$
Then $\widetilde{H}_1,...,\widetilde{H}_n$ are pairwise orthogonal and
$\varphi(\widetilde{H}_i)=H_i$, $i=1,...,n$.
Thus the system $(H;H_1,...,H_n)$ is isomorphic to the system
$(H_1\oplus...\oplus H_n\oplus N;\widetilde{H}_1,...,\widetilde{H}_n)$.

$(5)\Rightarrow (6)$
Assume that there exist a Hilbert space $K$ and pairwise orthogonal closed
subspaces $K_1,...,K_n$ of the space $K$ such that the system
$(H;H_1,...,H_n)$ is isomorphic to $(K;K_1,...,K_n)$.
Let $\varphi:H\to K$ be an isomorphism such that $\varphi(H_i)=K_i$,
$i=1,2,...,n$.
Define a new inner product $\langle\cdot,\cdot\rangle_0$ in $H$ by
$\langle x,y\rangle_0:=\langle\varphi(x),\varphi(y)\rangle_K$, $x,y\in H$.
It is clear that $\langle\cdot,\cdot\rangle_0$ is an inner product in $H$,
$\langle\cdot,\cdot\rangle_0$ is equivalent to the original inner product in $H$
and the subspaces $H_1,...,H_n$ are pairwise orthogonal in 
$\langle\cdot,\cdot\rangle_0$.

$(6)\Rightarrow(7)$
Assume that there exists an inner product $\langle\cdot,\cdot\rangle_0$ in $H$
which is equivalent to the original inner product in $H$ and such that
the subspaces $H_1,...,H_n$ are pairwise orthogonal in 
$(H;\langle\cdot,\cdot\rangle_0)$.
Let $E_i$ be the orthogonal projection onto $H_i$ in $(H;\langle\cdot,\cdot\rangle_0)$,
$i=1,...,n$.
Consider $E_1,...,E_n$ as operators from $H$ (with the original inner product)
to $H$ (with the original inner product).
It is clear that for every $i=1,...,n$ $E_i$ is a continuous linear projection
onto $H_i$ and $E_i E_j=0$ for each pair $i\neq j$.

$(7)\Rightarrow(2)$
Assume that there exist continuous linear projections $E_k:H\to H$, $k=1,...,n$,
such that $E_i E_j=0$ for every pair $i\neq j$ and $Ran(E_k)=H_k$ for
every $k=1,...,n$.
Let us show that $H_1,...,H_n$ are linearly independent.
Suppose that $x_1+...+x_n=0$, where $x_1\in H_1,...,x_n\in H_n$.
Consider arbitrary $i\in\{1,2,...,n\}$.
Then $E_i(x_1+...+x_n)=0$.
Since $E_i x_j=E_i E_j x_j=0$ for $j\neq i$ and $E_i x_i=x_i$,
we see that $x_i=0$.
Thus $H_1,...,H_n$ are linearly independent.
Let us show that $H_1+...+H_n$ is closed in $H$.
It is easy to see that $E:=E_1+...+E_n$ is a continuous linear projection
onto the subspace $H_1+...+H_n$.
Thus $H_1+...+H_n=\ker(I-E)$ is closed in $H$.

$(3)\Rightarrow(8)$
Assume that
$$
c:=\inf\{
\dfrac{\|x_1+...+x_n\|}{\sqrt{\|x_1\|^2+...+\|x_n\|^2}}\,|\,
x_1\in H_1,...,x_n\in H_n,
(x_1,...,x_n)\neq (0,...,0)\}>0.
$$
Then 
\begin{equation}\label{ineq:c}
\|x_1+...+x_n\|\geqslant c\sqrt{\|x_1\|^2+...+\|x_n\|^2}
\end{equation}
for arbitrary $x_1\in H_1,...,x_n\in H_n$.
Consider arbitrary $i\in\{1,...,n\}$.
If $x\in H_i$, $\|x\|=1$ and $y\in\sum_{j\neq i}H_j$, then inequality~\eqref{ineq:c}
implies that $\|x-y\|\geqslant c$.
It follows that $dist(x,\sum_{j\neq i}H_j)\geqslant c$ and consequently
$\delta(H_i,\sum_{j\neq i}H_j)\geqslant c>0$.

$(8)\Rightarrow(3)$
Assume that $\delta_i:=\delta(H_i,\sum_{j\neq i}H_j)>0$ for every $i=1,...,n$.
Consider arbitrary elements $x_1\in H_1,...,x_n\in H_n$.
Then $\|x_1+...+x_n\|\geqslant dist(x_i,\sum_{j\neq i}H_j)
\geqslant\delta_i\|x_i\|$, $i=1,...,n$.
Thus $\|x_i\|\leqslant\delta_i^{-1}\|x_1+...+x_n\|$, $i=1,...,n$
and consequently
$\|x_1\|^2+...+\|x_n\|^2\leqslant (\delta_1^{-2}+...+\delta_n^{-2})
\|x_1+...+x_n\|^2$.
It follows that
$$
\|x_1+...+x_n\|\geqslant\dfrac{1}{\sqrt{\delta_1^{-2}+...+\delta_n^{-2}}}
\sqrt{\|x_1\|^2+...+\|x_n\|^2}.
$$

$(5)\Rightarrow (9)$
Assume that there exist a Hilbert space $K$ and pairwise orthogonal closed
subspaces $K_1,...,K_n$ of the space $K$ such that the system
$(H;H_1,...,H_n)$ is isomorphic to $(K;K_1,...,K_n)$.
Let $\varphi:K\to H$ be an isomorphism such that 
$H_i=\varphi(K_i)$, $i=1,2,...,n$.
Define the closed subspace $N$ of the space $K$ by
$N:=K\ominus (K_1+...+K_n)$.
Then the subspaces $K_1,...,K_n,N$ are pairwise orthogonal and
their sum is equal to $K$.

Let us show that $\cap_{j\neq i}H_j^\bot=(\varphi^*)^{-1}(K_i+N)$,
$i=1,2,...,n$.
Consider arbitrary $i\in\{1,2,...,n\}$.
The subspace $\cap_{j\neq i}H_j^\bot$ consists of all elements $x\in H$
such that $x\bot H_j$ for every $j\neq i$.
This means that $\langle x,y\rangle=0$, $y\in H_j$, $j\neq i$
$\Leftrightarrow$
$\langle x,\varphi(z)\rangle=0$, $z\in K_j$, $j\neq i$
$\Leftrightarrow$
$\langle\varphi^*(x),z\rangle=0$, $z\in K_j$, $j\neq i$
$\Leftrightarrow$
$\varphi^*(x)\bot K_j$, $j\neq i$.
Now recall that the subspaces $K_1,...,K_n,N$ are pairwise orthogonal and
their sum is equal to $K$.
Therefore the latter condition on $x$ is equivalent to the inclusion
$\varphi^*(x)\in K_i+N$
$\Leftrightarrow$
$x\in(\varphi^*)^{-1}(K_i+N)$.
Thus $\cap_{j\neq i}H_j^\bot=(\varphi^*)^{-1}(K_i+N)$.
Now we see that 
$$
\sum_{i=1}^n\bigcap_{j\neq i}H_j^\bot=
\sum_{i=1}^n(\varphi^*)^{-1}(K_i+N)=
(\varphi^*)^{-1}(K_1+...+K_n+N)=
(\varphi^*)^{-1}(K)=H.
$$ 

$(9)\Rightarrow (10)$
Assume that $\sum_{i=1}^n\bigcap_{j\neq i}H_j^\bot=H$.
Consider arbitrary index $k\in\{1,2,...,n\}$.
Since $\sum_{i=1}^n\bigcap_{j\neq i}H_j^\bot\subset H_k^\bot+\cap_{j\neq k}H_j^\bot$,
we conclude that $H_k^\bot+\cap_{j\neq k}H_j^\bot=H$.

$(10)\Rightarrow (1)$
Assume that $H_i^\bot+\cap_{j\neq i}H_j^\bot=H$ for every $i=1,2,...,n$.
We have to show that the system $H_1,...,H_n$ possesses the IBAP.
Consider arbitrary elements $x_i\in H_i$, $i=1,2,...,n$.
Since $H_i^\bot+\cap_{j\neq i}H_j^\bot=H$, we conclude that there exist
elements $y_i\in H_i^\bot$ and $z_i\in\cap_{j\neq i}H_j^\bot$
such that $x_i=y_i+z_i$.
Then $z_i=x_i-y_i$ and consequently $P_i z_i=x_i$.
Also $P_j z_i=0$ for every $j\neq i$.
Define $z:=z_1+...+z_n$ .
It is clear that $P_i z=x_i$, $i=1,2,...,n$.
This means that the system of subspaces $H_1,...,H_n$ possesses the IBAP.
\end{proof}

\section{Further results on the IBAP}\label{S:further results}

\subsection{IBAP and Riesz families}

The following theorem is motivated by~\cite[Theorem~2.1]{T00}.

\begin{theorem}\label{T:IBAP and Riesz}
Let $H$ be a Hilbert space and $H_1,...,H_n$ be closed subspaces of $H$.
The following statements are true:

(1) Assume that the system of subspaces $H_1,...,H_n$ possesses the IBAP.
If $\{v^{(i)}_\alpha|\alpha\in M_i\}$ is a Riesz family in $H_i$, $i=1,2,...,n$,
then the family $\{v^{(i)}_\alpha|i=1,2,...,n, \alpha\in M_i\}$ is also a Riesz family.

(2) Assume that for every $i=1,2,...,n$ a family of elements 
$\{v^{(i)}_\alpha|\alpha\in M_i\}$ of the subspace $H_i$
is such that the linear span of the family is dense in $H_i$.
If the family of elements $\{v^{(i)}_\alpha|i=1,2,...,n, \alpha\in M_i\}$ 
is a Riesz family, then the system of subspaces $H_1,...,H_n$ possesses the IBAP. 
\end{theorem}
\begin{proof}
(1)
By Theorem~\ref{T:main}, $(1)\Leftrightarrow (3)$, there exists a number
$\varepsilon>0$ such that
$$
\|x_1+...+x_n\|\geqslant\varepsilon(\|x_1\|^2+...+\|x_n\|^2)^{1/2}
$$
for arbitrary elements $x_1\in H_1,...,x_n\in H_n$.
Since $\{v^{(i)}_\alpha|\alpha\in M_i\}$ is a Riesz family, we conclude that there
exist numbers $\varepsilon_i>0$ and $C_i>0$ such that
$$
\varepsilon_i
\left(\sum_{\alpha\in F_i}|a^{(i)}_\alpha|^2\right)^{1/2}\leqslant
\|\sum_{\alpha\in F_i}a^{(i)}_\alpha v^{(i)}_\alpha\|\leqslant
C_i 
\left(\sum_{\alpha\in F_i}|a^{(i)}_\alpha|^2\right)^{1/2}
$$
for every finite subset $F_i\subset M_i$ and each collection of scalars
$\{a^{(i)}_\alpha|\alpha\in F_i\}$.
Let us show that the family
$\{v^{(i)}_\alpha|i=1,2,...,n, \alpha\in M_i\}$
is a Riesz family.
Consider arbitrary finite subfamily of the family
$\{v^{(i)}_\alpha|(i,\alpha)\in\cup_{k=1}^n\{k\}\times M_k\}$.
The subfamily has a form
$\{v^{(i)}_\alpha|(i,\alpha)\in\cup_{k=1}^n\{k\}\times F_k\}$,
where $F_k$ is a finite subset of $M_k$, $k=1,2,...,n$.
For arbitrary collection of scalars
$\{a^{(i)}_\alpha|(i,\alpha)\in\cup_{k=1}^n\{k\}\times F_k\}$
we have
\begin{align*}
&\|\sum_{\alpha\in F_1}a^{(1)}_\alpha v^{(1)}_\alpha+...+
\sum_{\alpha\in F_n}a^{(n)}_\alpha v^{(n)}_\alpha\|\geqslant\\
&\geqslant
\varepsilon\left(\|\sum_{\alpha\in F_1}a^{(1)}_\alpha v^{(1)}_\alpha\|^2+...+
\|\sum_{\alpha\in F_n}a^{(n)}_\alpha v^{(n)}_\alpha\|^2\right)^{1/2}\geqslant\\
&\geqslant
\varepsilon\left(\varepsilon_1^2\sum_{\alpha\in F_1}|a^{(1)}_\alpha|^2+...+
\varepsilon_n^2\sum_{\alpha\in F_n}|a^{(n)}_\alpha|^2\right)^{1/2}\geqslant\\
&\geqslant
\varepsilon\min\{\varepsilon_1,...,\varepsilon_n\}
\left(\sum_{\alpha\in F_1}|a^{(1)}_\alpha|^2+...+
\sum_{\alpha\in F_n}|a^{(n)}_\alpha|^2\right)^{1/2}.
\end{align*}
Also
\begin{align*}
&\|\sum_{\alpha\in F_1}a^{(1)}_\alpha v^{(1)}_\alpha+...+
\sum_{\alpha\in F_n}a^{(n)}_\alpha v^{(n)}_\alpha\|\leqslant\\
&\leqslant
\|\sum_{\alpha\in F_1}a^{(1)}_\alpha v^{(1)}_\alpha\|+...+
\|\sum_{\alpha\in F_n}a^{(n)}_\alpha v^{(n)}_\alpha\|\leqslant\\
&\leqslant
C_1\left(\sum_{\alpha\in F_1}|a^{(1)}_\alpha|^2\right)^{1/2}+...+
C_n\left(\sum_{\alpha\in F_n}|a^{(n)}_\alpha|^2\right)^{1/2}\leqslant\\
&\leqslant
(C_1^2+...+C_n^2)^{1/2}
\left(\sum_{\alpha\in F_1}|a^{(1)}_\alpha|^2+...+
\sum_{\alpha\in F_n}|a^{(n)}_\alpha|^2\right)^{1/2}.
\end{align*}
Thus $\{v^{(i)}_\alpha|(i,\alpha)\in\cup_{k=1}^n\{k\}\times M_k\}$
is a Riesz family.

(2)
Since the family of elements $\{v^{(i)}_\alpha|i=1,2,...,n, \alpha\in M_i\}$ 
is a Riesz family, we conclude that there exist numbers $\varepsilon>0$ and $C>0$
such that
\begin{align*}
\varepsilon
\left(\sum_{\alpha\in F_1}|a^{(1)}_\alpha|^2+...+
\sum_{\alpha\in F_n}|a^{(n)}_\alpha|^2\right)^{1/2}&\leqslant
\|\sum_{\alpha\in F_1}a^{(1)}_\alpha v^{(1)}_\alpha+...+
\sum_{\alpha\in F_n}a^{(n)}_\alpha v^{(n)}_\alpha\|\leqslant\\
&\leqslant C
\left(\sum_{\alpha\in F_1}|a^{(1)}_\alpha|^2+...+
\sum_{\alpha\in F_n}|a^{(n)}_\alpha|^2\right)^{1/2}
\end{align*}
for every finite subsets $F_1\subset M_1,...,F_n\subset M_n$
and each collection of scalars
$\{a^{(i)}_\alpha|i=1,2,...,n,\alpha\in F_i\}$.
It follows that for every $i=1,2,...,n$ the family
$\{v^{(i)}_\alpha|\alpha\in M_i\}$ is also a Riesz family with
lower bound $\varepsilon$ and upper bound $C$, i.e.,
$$
\varepsilon
\left(\sum_{\alpha\in F_i}|a^{(i)}_\alpha|^2\right)^{1/2}\leqslant
\|\sum_{\alpha\in F_i}a^{(i)}_\alpha v^{(i)}_\alpha\|\leqslant
C 
\left(\sum_{\alpha\in F_i}|a^{(i)}_\alpha|^2\right)^{1/2}
$$
for every finite subset $F_i\subset M_i$ and each collection of scalars
$\{a^{(i)}_\alpha|\alpha\in F_i\}$.
Denote by $H_i^0$ the linear span of $\{v^{(i)}_\alpha|\alpha\in M_i\}$.
Consider arbitrary elements $x_1\in H_1^0,...,x_n\in H_n^0$.
Then $x_i=\sum_{\alpha\in F_i}a^{(i)}_\alpha v^{(i)}_\alpha$
for some finite subset $F_i\subset M_i$ and some
collection of scalars $\{a^{(i)}_\alpha|\alpha\in F_i\}$, $i=1,2,...,n$.
We have
\begin{align*}
\|x_1+...+x_n\|&=
\|\sum_{\alpha\in F_1}a^{(1)}_\alpha v^{(1)}_\alpha+...+
\sum_{\alpha\in F_n}a^{(n)}_\alpha v^{(n)}_\alpha\|\geqslant\\
&\geqslant
\varepsilon
\left(\sum_{\alpha\in F_1}|a^{(1)}_\alpha|^2+...+
\sum_{\alpha\in F_n}|a^{(n)}_\alpha|^2\right)^{1/2}\geqslant\\
&\geqslant
\varepsilon
\left(C^{-2}\|\sum_{\alpha\in F_1}a^{(1)}_\alpha v^{(1)}_\alpha\|^2+...+
C^{-2}\|\sum_{\alpha\in F_n}a^{(n)}_\alpha v^{(n)}_\alpha\|^2\right)^{1/2}=\\
&=\varepsilon C^{-1}(\|x_1\|^2+...+\|x_n\|^2)^{1/2}.
\end{align*}
Therefore
\begin{equation}\label{ineq:eC-1}
\|x_1+...+x_n\|\geqslant\varepsilon C^{-1}(\|x_1\|^2+...+\|x_n\|^2)^{1/2}
\end{equation}
for arbitrary $x_1\in H_1^0,...,x_n\in H_n^0$.
Since $H_i^0$ is dense in $H_i$ for every $i=1,2,...,n$, we
conclude that inequality~\eqref{ineq:eC-1} holds for arbitrary
$x_1\in H_1,...,x_n\in H_n$.
By Theorem~\ref{T:main}, $(1)\Leftrightarrow (3)$,
the system of subspaces $H_1,...,H_n$ possesses the IBAP.
\end{proof}

\subsection{IBAP and eigenspaces}

\begin{theorem}\label{T:eigenspaces}
Let $H$ be a Hilbert space.
The following statements are true:

(1) if $A:H\to H$ is a continuous linear operator and
scalars $\lambda_1,...,\lambda_n$ are pairwise distinct, then
the system of eigenspaces $H_k:=\ker(A-\lambda_k I)$, $k=1,...,n$
possesses the IBAP.

(2) if a system of closed subspaces $H_1,...,H_n$ possesses the IBAP and
scalars $\lambda_1,...,\lambda_n$ are pairwise distinct, then
there exists a continuous linear operator $A:H\to H$ such that
$H_k=\ker(A-\lambda_k I)$, $k=1,...,n$.
\end{theorem}
\begin{proof}
(1) To prove that the system $H_1,...,H_n$ possesses the IBAP we will
use Theorem~\ref{T:main}, $(1)\Leftrightarrow (8)$.
Consider arbitrary index $i\in\{1,...,n\}$.
Let us prove that $\delta(H_i,\sum_{j\neq i}H_j)>0$.
If $H_i=\{0\}$, then this is clear.
Assume that $H_i\neq\{0\}$.
Consider the operator $T_i:=\prod_{j\neq i}(A-\lambda_j I)/
\prod_{j\neq i}(\lambda_i-\lambda_j)$.
It is clear that $T_i x=x$ for $x\in H_i$ and
$T_i y=0$ for $y\in H_j$, $j\neq i$.
Now consider arbitrary $x\in H_i$ with $\|x\|=1$.
For every $y\in\sum_{j\neq i}H_j$ we have
$$
\|x-y\|\geqslant\dfrac{\|T_i(x-y)\|}{\|T_i\|}=\dfrac{\|x-0\|}{\|T_i\|}=
\dfrac{1}{\|T_i\|}.
$$
It follows that $dist(x,\sum_{j\neq i}H_j)\geqslant 1/\|T_i\|$,
and consequently $\delta(H_i,\sum_{j\neq i}H_j)\geqslant 1/\|T_i\|>0$.

(2)
By Theorem~\ref{T:main}, $(1)\Leftrightarrow (2)$
the subspaces $H_1,...,H_n$ are linearly independent and their sum
is closed in $H$.
Define the closed subspace $H_{n+1}:=H\ominus (H_1+...+H_n)$
and consider the system of subspaces $H_1,...,H_n,H_{n+1}$.
The subspaces are linearly independent and their sum is equal to $H$.
By Theorem~\ref{T:main}, $(2)\Leftrightarrow (7)$, there exist
continuous linear projections $E_1,...,E_{n+1}$ such that
$E_i E_j=0$ for every pair $i\neq j$ and $H_k=Ran(E_k)$, $k=1,...,n+1$.
Take arbitrary scalar $\lambda_{n+1}$ which is different from 
$\lambda_1,...,\lambda_n$ and consider the operator
$A:=\lambda_1 E_1+...+\lambda_{n+1}E_{n+1}$.
In fact, if $x\in H$ and $x=x_1+...+x_{n+1}$, where $x_i\in H_i$, $i=1,...,n+1$,
then $Ax=\lambda_1 x_1+...+\lambda_{n+1}x_{n+1}$.
Thus $\ker(A-\lambda_k I)=H_k$ for $k=1,...,n,n+1$.
\end{proof}

\subsection{IBAP and root subspaces}

\begin{theorem}\label{T:root subspaces}
Let $H$ be a Hilbert space,
$A:H\to H$ be a continuous linear operator,
$\lambda_1,...,\lambda_n$ be pairwise distinct scalars,
and $m_1,...,m_n$ be natural numbers.
Then the system of subspaces $H_k:=\ker(A-\lambda_k I)^{m_k}$,
$k=1,...,n$ possesses the IBAP.
\end{theorem}
\begin{proof}
To prove that the system $H_1,...,H_n$ possesses the IBAP we will
use Theorem~\ref{T:main}, $(1)\Leftrightarrow (8)$.
Consider the polynomials $\prod_{j\neq i}(\lambda-\lambda_j)^{m_j}$,
$i=1,...,n$.
Since the greatest common divisor of these polynomials is equal to $1$,
we conclude that there exist polynomials $p_1(\lambda),...,p_n(\lambda)$
such that 
$\sum_{i=1}^n p_i(\lambda)\prod_{j\neq i}(\lambda-\lambda_j)^{m_j}=1$.
Define the operators $T_i:=p_i(A)\prod_{j\neq i}(A-\lambda_j I)^{m_j}$,
$i=1,...,n$.
Then $T_1+...+T_n=I$ and $T_i|_{H_j}=0$ for every pair $j\neq i$.
Consequently, $T_i x=x$ for $x\in H_i$.
Similarly to the proof of Theorem~\ref{T:eigenspaces}, (1),
one can show that $\delta(H_i,\sum_{j\neq i}H_j)\geqslant 1/\|T_i\|>0$.
\end{proof}

\subsection{Stability of the IBAP}

\begin{theorem}\label{T:stability}
Let $H$ be a Hilbert space and $H_1,...,H_n$ be closed subspaces of $H$
such that $H_k\neq\{0\}$ for some $k$.
Suppose that the system of subspaces $H_1,...,H_n$ possesses the IBAP.
Set
$$
c:=\inf\{
\dfrac{\|x_1+...+x_n\|}{\sqrt{\|x_1\|^2+...+\|x_n\|^2}}\,|\,
x_1\in H_1,...,x_n\in H_n,
(x_1,...,x_n)\neq (0,...,0)\}.
$$
If a system of closed subspaces $H_1',...,H_n'$ is such that
$$
(\theta(H_1,H_1'))^2+...+(\theta(H_n,H_n'))^2<c^2,
$$
then the system $H_1',...,H_n'$ also possesses the IBAP.
\end{theorem}

To prove Theorem~\ref{T:stability} we need the following lemma.

\begin{lemma}\label{L:gap<1}
Let $H$ be a Hilbert space and $M,N$ be two closed subspaces of $H$.
Denote by $P$ the orthogonal projection onto $M$.
If $\theta(M,N)<1$, then $P(N)=M$.
\end{lemma}
\begin{proof}
See \cite[Proof of Lemma 3.1]{GK69}.
\end{proof}

\begin{proof}[Proof of Theorem~\ref{T:stability}]
Set $\theta_k:=\theta(H_k,H_k')$, $k=1,2,...,n$.
Define the operator $S:H_1\oplus...\oplus H_n\to H$ by
$$
S(x_1,...,x_n)^t:=x_1+...+x_n,\quad x_1\in H_1,...,x_n\in H_n.
$$
Then $\|Sv\|\geqslant c\|v\|$, $v\in H_1\oplus...\oplus H_n$.
Denote by $P_k$ the orthogonal projection onto $H_k$ and
by $P_k'$ the orthogonal projection onto $H_k'$, $k=1,2,...,n$.
Define the operator $T:H_1\oplus...\oplus H_n\to H$ by
$$
T(x_1,...,x_n)^t:=(P_1'-P_1)x_1+...+(P_n'-P_n)x_n, \quad x_1\in H_1,...,x_n\in H_n.
$$
Then for arbitrary element $v=(x_1,...,x_n)^t\in H_1\oplus...\oplus H_n$
we have
\begin{align*}
\|Tv\|&\leqslant\|P_1'-P_1\|\|x_1\|+...+\|P_n'-P_n\|\|x_n\|=
\theta_1\|x_1\|+...+\theta_n\|x_n\|\leqslant\\
&\leqslant\sqrt{\theta_1^2+...+\theta_n^2}\sqrt{\|x_1\|^2+...+\|x_n\|^2}=
\sqrt{\theta_1^2+...+\theta_n^2}\|v\|.
\end{align*}
It follows that
$$
\|(S+T)v\|=\|Sv+Tv\|\geqslant\|Sv\|-\|Tv\|\geqslant
(c-\sqrt{\theta_1^2+...+\theta_n^2})\|v\|.
$$
Therefore the operator $S+T$ is an isomorphic embedding.
For every $k=1,...,n$ define the subspace 
$\widetilde{H}_k\subset H_1\oplus...\oplus H_n$ by
$$
\widetilde{H}_k:=
\{0\}\oplus...\oplus\{0\}\oplus H_k\oplus\{0\}\oplus...\oplus\{0\}=
\{(0,...,0,x_k,0,...,0)|x_k\in H_k\}.
$$
If $v=(0,...,0,x_k,0,...,0)\in\widetilde{H}_k$, then
$(S+T)v=x_k+(P_k'-P_k)x_k=P_k'(x_k)$.
Thus $(S+T)(\widetilde{H}_k)=P_k'(H_k)$.
Since $\theta(H_k',H_k)=\theta_k<c\leqslant 1$,
by Lemma~\ref{L:gap<1} we conclude that $P_k'(H_k)=H_k'$.
Therefore $(S+T)(\widetilde{H}_k)=H_k'$, $k=1,2,...,n$.
Since the subspaces $\widetilde{H}_1,...,\widetilde{H}_n$
are linearly independent and $S+T$ is an isomorphic embedding,
we see that the subspaces $H_1',...,H_n'$ are linearly independent.
Moreover, their sum
$H_1'+...+H_n'=(S+T)(\widetilde{H}_1)+...+(S+T)(\widetilde{H}_n)=
(S+T)(\widetilde{H}_1+...+\widetilde{H}_n)=(S+T)(H_1\oplus...\oplus H_n)$
is closed in $H$.
By Theorem~\ref{T:main}, $(1)\Leftrightarrow(2)$,
the system of subspaces $H_1',...,H_n'$ possesses the IBAP.
\end{proof}

\section{Inverse marginal property}\label{S:IMP}

\subsection{Definitions}
Let $(\Omega,\mathcal{F},\mu)$ be a probability space.
Denote by $\mathbb{K}$ a base field of scalars, i.e., $\mathbb{R}$ or $\mathbb{C}$.
For an $\mathcal{F}$-measurable function (random variable) 
$\xi:\Omega\to\mathbb{K}$ denote by $E\xi$ the expectation of $\xi$ (if it exists).
Two random variables $\xi$ and $\eta$ are said to be equivalent 
if $\xi(\omega)=\eta(\omega)$ for $\mu$-almost all $\omega$.
Denote by $L^2(\mathcal{F})=L^2(\Omega,\mathcal{F},\mu)$ the set of equivalence classes of
random variables $\xi:\Omega\to\mathbb{K}$ such that $E|\xi|^2<\infty$. 
For $\xi,\eta\in L^2(\mathcal{F})$ set
$\langle\xi,\eta\rangle:=E(\xi\overline{\eta})$.
Then $L^2(\mathcal{F})$ is a Hilbert space.
For every sub-$\sigma$-algebra $\mathcal{A}$ of $\mathcal{F}$ we define the marginal subspace corresponding to $\mathcal{A}$, $L^2(\mathcal{A})$, as follows.
$L^2(\mathcal{A})$ consists of elements (equivalence classes) of 
$L^2(\mathcal{F})$ which contain at least one $\mathcal{A}$-measurable random variable.
The subspace $L^2(\mathcal{A})$ is closed in $L^2(\mathcal{F})$.
Indeed, it is easily seen that $L^2(\mathcal{A})$
is canonically isometrically isomorphic to 
$L^2(\Omega,\mathcal{A},\mu|_{\mathcal{A}})$.
Since the latter space is Hilbert, we conclude that $L^2(\mathcal{A})$
is complete.
It follows that $L^2(\mathcal{A})$ is closed in $L^2(\mathcal{F})$.
Moreover, it is well-known that
the conditional expectation operator $\xi\mapsto E(\xi|\mathcal{A})$ 
is the orthogonal projection onto $L^2(\mathcal{A})$ in $L^2(\mathcal{F})$.
Denote by $L^2_0(\mathcal{A})$ the subspace of all 
$\xi\in L^2(\mathcal{A})$ with $E\xi=0$.
This subspace is also closed in $L^2(\mathcal{F})$.
This follows from the equality 
$L^2_0(\mathcal{A})=L^2(\mathcal{A})\cap\{1\}^\bot$. 

\subsection{The inverse marginal property}

Let $(\Omega,\mathcal{F},\mu)$ be a probability space and 
$\mathcal{F}_1,...,\mathcal{F}_n$ be sub-$\sigma$-algebras of $\mathcal{F}$.
We will say that the collection $\mathcal{F}_1,...,\mathcal{F}_n$ possesses
the inverse marginal property (IMP) if for arbitrary random variables
$\xi_1,...,\xi_n$ such that 

(1) $\xi_k$ is $\mathcal{F}_k$-measurable, $k=1,2,...,n$;

(2) $E|\xi_k|^2<\infty$, $k=1,2,...,n$;

(3) $E\xi_1=E\xi_2=...=E\xi_n$,

there exists a random variable $\xi$ such that $E|\xi|^2<\infty$ and 
$E(\xi|\mathcal{F}_k)=\xi_k$ for all $k=1,2,...,n$.

The simplest example of a collection of sub-$\sigma$-algebras which possesses the IMP 
is a system of pairwise independent sub-$\sigma$-algebras.
In this case a needed random variable $\xi$ 
can be defined by $\xi:=\xi_1+...+\xi_n-(n-1)a$, where
$a:=E\xi_1=E\xi_2=...=E\xi_n$.

\begin{theorem}\label{T:IMP}
The collection of sub-$\sigma$-algebras 
$\mathcal{F}_1,...,\mathcal{F}_n$ possesses the IMP
if and only if
the system of subspaces $L^2_0(\mathcal{F}_1),...,L^2_0(\mathcal{F}_n)$ 
possesses the IBAP (in $L^2(\mathcal{F})$).
\end{theorem}
\begin{proof}
Denote by $P_k$ the orthogonal projection onto $L^2_0(\mathcal{F}_k)$, $k=1,...,n$.
Since $L^2_0(\mathcal{F}_k)=L^2(\mathcal{F}_k)\ominus span(1)$
(here $span(1)$ is the subspace spanned by $1$, i.e., the subspace
of constant random variables), we conclude that
$P_k\xi=E(\xi|\mathcal{F}_k)-E\xi$, $\xi\in L^2(\mathcal{F})$.

First assume that the collection $\mathcal{F}_1,...,\mathcal{F}_n$ possesses the IMP.
We have to show that the system of subspaces
$L^2_0(\mathcal{F}_1),...,L^2_0(\mathcal{F}_n)$ possesses the IBAP.
Consider arbitrary elements $\xi_k\in L^2_0(\mathcal{F}_k)$,
$k=1,2,...,n$.
Since the collection $\mathcal{F}_1,...,\mathcal{F}_n$
possesses the IMP, we conclude that there exists an element
$\xi\in L^2(\mathcal{F})$ such that $E(\xi|\mathcal{F}_k)=\xi_k$
for all $k=1,2,...,n$.
Clearly, $E\xi=0$.
Thus $P_k\xi=E(\xi|\mathcal{F}_k)-E\xi=\xi_k$, $k=1,2,...,n$.
Therefore the system of subspaces
$L^2_0(\mathcal{F}_1),...,L^2_0(\mathcal{F}_n)$ possesses the IBAP.

Now assume that the system of subspaces
$L^2_0(\mathcal{F}_1),...,L^2_0(\mathcal{F}_n)$ possesses the IBAP.
We have to show that the collection of sub-$\sigma$-algebras
$\mathcal{F}_1,...,\mathcal{F}_n$ possesses the IMP.
Consider arbitrary random variables $\xi_1,...,\xi_n$ such that 

(1) $\xi_k$ is $\mathcal{F}_k$-measurable, $k=1,2,...,n$;

(2) $E|\xi_k|^2<\infty$, $k=1,2,...,n$;

(3) $E\xi_1=E\xi_2=...=E\xi_n$. 

Set $a:=E\xi_1=E\xi_2=...=E\xi_n$.
Then $\xi_k-a\in L^2_0(\mathcal{F}_k)$, $k=1,2,...,n$.
Since the system of subspaces
$L^2_0(\mathcal{F}_1),...,L^2_0(\mathcal{F}_n)$ possesses the IBAP,
we conclude that there exists an element $\xi\in L^2(\mathcal{F})$
such that $P_k\xi=\xi_k-a$ for all $k=1,2,...,n$.
This means that $E(\xi|\mathcal{F}_k)-E\xi=\xi_k-a$, $k=1,2,...,n$.
It follows that $E(\xi-E\xi+a|\mathcal{F}_k)=\xi_k$, $k=1,2,...,n$.
Therefore the collection of sub-$\sigma$-algebras
$\mathcal{F}_1,...,\mathcal{F}_n$ possesses the IMP.
\end{proof}

\subsection{On the closedness of the sum of marginal subspaces}

Theorem~\ref{T:IMP} and Theorem~\ref{T:main}, $(1)\Leftrightarrow (2)$,
show that a collection of sub-$\sigma$-algebras $\mathcal{F}_1,...,\mathcal{F}_n$
possesses the IMP if and only if the subspaces
$L^2_0(\mathcal{F}_1),...,L^2_0(\mathcal{F}_n)$ are linearly independent
and their sum is closed in $L^2(\mathcal{F})$.
One can easily check that the subspace 
$L^2_0(\mathcal{F}_1)+...+L^2_0(\mathcal{F}_n)$ is closed in $L^2(\mathcal{F})$
if and only if the subspace
$L^2(\mathcal{F}_1)+...+L^2(\mathcal{F}_n)$ is closed in $L^2(\mathcal{F})$.
It is worth mentioning that the problem on the closedness of the sum
of marginal subspaces is important and arises, for example, in

(1) statistics (see, e.g., \cite{Bickel91});

(2) additive modeling (see, e.g., \cite{Buja96}).
Here each sub-$\sigma$-algebra $\mathcal{F}_i=\sigma a(\xi_i)$, 
the $\sigma$-algebra generated by a random variable $\xi_i$.
Consequently, each marginal subspace $L^2(\mathcal{F}_i)$
consists of (equivalence classes of) Borel measurable transformations of $\xi_i$, 
$f(\xi_i)$, which belong to $L^2(\mathcal{F})$.
As Andreas Buja writes in \cite[Subsection 8.1]{Buja96},
the question on closedness of $L^2_0(\mathcal{F}_1)+...+L^2_0(\mathcal{F}_n)$ is a technicality
that is at the heart of all additive modeling, 
including ACE (alternating conditional expectations method),
GAMs (generalized additive models) and PPR (projection pursuit regression).

(3) theoretical tomography and theory of ridge functions
(see, e.g., \cite[Introduction and Chapter 7]{Pinkus15}).
Note that every subspace of ridge functions $L^2(a;K)$ can be considered as marginal. 

\subsection{Example 1 ($\Omega=\mathbb{N}$)}\label{SS:Omega=N}

Consider the space $\Omega:=\mathbb{N}=\{1,2,...\}$
with the $\sigma$-algebra $\mathcal{F}:=2^{\mathbb{N}}$.
Then a probability measure $\mu$ is defined by a set of numbers
$\mu(\{k\})=p_k\geqslant 0$, $k=1,2,...$ with $\sum_{k=1}^\infty p_k=1$.
In what follows we assume that $p_k>0$, $k=1,2,...$.
Consider the Hilbert space $H:=L^2(\mathbb{N},2^{\mathbb{N}},\mu)$.
The space consists of all functions (sequences) $f:\mathbb{N}\to\mathbb{K}$
such that $\int_{\mathbb{N}}|f(k)|^2 d\mu(k)=\sum_{k=1}^\infty p_k|f(k)|^2<\infty$.

Let $A$ be an infinite subset of $\mathbb{N}\setminus\{1\}$.
Define the partition of $\mathbb{N}$, $part(A)$, as follows.
Let $A=\{k_1,k_2,k_3,...\}$, where $2\leqslant k_1<k_2<k_3<...$.
Set
$$
part(A)=\{\{1,...,k_1-1\},\{k_1,...,k_2-1\},\{k_2,...,k_3-1\},...\}.
$$
The numbers $1,k_1,k_2,...$ are ``starting points'' for the sets of $part(A)$.
Let $\mathcal{A}=\sigma a(A)$ be the $\sigma$-algebra generated by the
partition $part(A)$.
Then the marginal subspace $L^2(\mathcal{A})$ consists of all functions $f\in H$
which are constant on each of the elements of $part(A)$, i.e., such that
$$
f(1)=...=f(k_1-1), f(k_1)=...=f(k_2-1), f(k_2)=...=f(k_3-1),....
$$

We will study the following problem.
Let $A_1,...,A_n$ be infinite subsets of $\mathbb{N}\setminus\{1\}$.
Denote by $\mathcal{A}_1,...,\mathcal{A}_n$ the corresponding $\sigma$-algebras.
Question: when the collection of $\sigma$-algebras $\mathcal{A}_1,...,\mathcal{A}_n$
possesses the IMP?
Equivalently, when the system of subspaces
$L^2_0(\mathcal{A}_1),...,L^2_0(\mathcal{A}_n)$ possesses the IBAP?
In the following theorem we establish relations between 
the IMP, the IBAP, closedness of the sum of marginal subspaces and
``fast decreasing'' of tails of the series $\sum_{k=1}^\infty p_k$.

\begin{theorem}\label{T:N}
Suppose that $\mathbb{N}\setminus\{1\}$ is partitioned into
$n$ infinite subsets $A_1,...,A_n$.
Let $\mathcal{A}_1,...,\mathcal{A}_n$ be the corresponding $\sigma$-algebras.
Set $r_k:=\sum_{j=k}^\infty p_j$, $k\geqslant 1$.
Consider the following statements:

(1) the collection $\mathcal{A}_1,...,\mathcal{A}_n$ possesses the IMP;

(2) the system of subspaces $L^2_0(\mathcal{A}_1),...,L^2_0(\mathcal{A}_n)$
possesses the IBAP;

(3) the subspace $L^2_0(\mathcal{A}_1)+...+L^2_0(\mathcal{A}_n)$ 
is closed in $L^2(2^{\mathbb{N}})$;

(4) $L^2_0(\mathcal{A}_1)+...+L^2_0(\mathcal{A}_n)=L^2_0(2^{\mathbb{N}})$;

(5) $\sup\{r_{k+1}/r_k\,|\,k\geqslant 1\}<1$.

Then statements (1), (2), (3), (4) are pairwise equivalent and $(5)\Rightarrow (1)$.
If there exists a number $k_0\geqslant 2$ such that for every $k\geqslant k_0$
the numbers $k$ and $k+1$ belong to distinct sets $A_i$, then
$(1)\Rightarrow(5)$.
\end{theorem}

\subsection{Auxiliary definitions}

To prove Theorem~\ref{T:N} we need a few auxiliary definitions:

1. Let $A$ be an infinite subset of $\mathbb{N}$.
We have already defined $part(A)$ and $\sigma a(A)$ if $1\notin A$.
Now let $1\in A$.
Let $A=\{k_1,k_2,k_3,...\}$, where $1=k_1<k_2<k_3<...$.
Define
$$
part(A)=\{\{k_1,...,k_2-1\},\{k_2,...,k_3-1\},\{k_3,...,k_4-1\},...\}
$$
and let $\mathcal{A}=\sigma a(A)$ be the $\sigma$-algebra generated by the
partition $part(A)$. 
It is clear that $part(A)=part(A\setminus\{1\})$
and $\sigma a(A)=\sigma a(A\setminus\{1\})$.

2. Let $A=\{k_1,k_2,k_3,...\}$ be an infinite subset of $\mathbb{N}$
(we assume that $k_1<k_2<k_3<...$)
and $\mathcal{A}=\sigma a(A)$ be the corresponding $\sigma$-algebra.
Define the closed subspace $H_A$ of $H$ as follows.
If $k_1=1$, then $H_A:=L^2(\mathcal{A})$;
if $k_1>1$, then $H_A:=\{f\in L^2(\mathcal{A})\,|\,f(1)=0\}$.
Note that in the latter case the subspace $H_A$ consists of all
functions $f\in H$ for which
$$
f(1)=...=f(k_1-1)=0, f(k_1)=...=f(k_2-1), f(k_2)=...=f(k_3-1),....
$$

\subsection{Extended version of Theorem~\ref{T:N}}

We will prove the following extended version of Theorem~\ref{T:N}.

\begin{theorem}\label{T:extended N}
Suppose that $\mathbb{N}$ is partitioned into $n$ infinite subsets $A_1,...,A_n$.
Let $\mathcal{A}_1,...,\mathcal{A}_n$ be the corresponding $\sigma$-algebras.
Set $r_k:=\sum_{j=k}^\infty p_j$, $k\geqslant 1$.
Consider the following statements:

(1) the collection $\mathcal{A}_1,...,\mathcal{A}_n$ possesses the IMP;

(2) the system of subspaces $L^2_0(\mathcal{A}_1),...,L^2_0(\mathcal{A}_n)$
possesses the IBAP;

(3) the system of subspaces $H_{A_1},...,H_{A_n}$ possesses the IBAP;

(4a) the subspace $L^2_0(\mathcal{A}_1)+...+L^2_0(\mathcal{A}_n)$ 
is closed in $L^2(2^{\mathbb{N}})$;

(4b) the subspace $L^2(\mathcal{A}_1)+...+L^2(\mathcal{A}_n)$ 
is closed in $L^2(2^{\mathbb{N}})$;

(5a) $L^2_0(\mathcal{A}_1)+...+L^2_0(\mathcal{A}_n)=L^2_0(2^{\mathbb{N}})$;

(5b) $L^2(\mathcal{A}_1)+...+L^2(\mathcal{A}_n)=L^2(2^{\mathbb{N}})$;

(6) $\sup\{r_{k+1}/r_k\,|\,k\geqslant 1\}<1$.

Then statements (1), (2), (3), (4a), (4b), (5a), (5b) are pairwise equivalent
and $(6)\Rightarrow (1)$.
If there exists a number $k_0\geqslant 1$ such that for every $k\geqslant k_0$
the numbers $k$ and $k+1$ belong to distinct sets $A_i$, then
$(1)\Rightarrow(6)$.
\end{theorem}

Two remarks are in order.

\begin{remark} 
In comparison with Theorem~\ref{T:N} Theorem~\ref{T:extended N}
contains one extra point (3).
This point is crucial for the proof of Theorem~\ref{T:extended N}.
\end{remark}

\begin{remark} 
Theorem~\ref{T:extended N} implies Theorem~\ref{T:N}
(just apply Theorem~\ref{T:extended N} to the sets
$A_1\cup\{1\},A_2,...,A_n$ and note that 
$\sigma a(A_1\cup\{1\})=\sigma a(A_1)$).
\end{remark}

\subsection{Auxiliary lemmas}

To prove Theorem~\ref{T:extended N} we need the following four lemmas.

\begin{lemma}\label{L:sum L2}
Suppose $A_1,...,A_n$ are infinite subsets of $\mathbb{N}$ such that
$A_1\cup...\cup A_n=\mathbb{N}$.
Let $\mathcal{A}_1,...,\mathcal{A}_n$ be the corresponding $\sigma$-algebras.
Then

(1) $H_{A_1}+...+H_{A_n}=L^2(\mathcal{A}_1)+...+L^2(\mathcal{A}_n)$;

(2) the subspace $L^2(\mathcal{A}_1)+...+L^2(\mathcal{A}_n)$ 
is dense in $H=L^2(2^{\mathbb{N}})$.
\end{lemma}
\begin{proof}
Let us prove (1).
For every $i$ we have $L^2(\mathcal{A}_i)=H_{A_i}+span(1)$.
Therefore $L^2(\mathcal{A}_1)+...+L^2(\mathcal{A}_n)=H_{A_1}+...+H_{A_n}+span(1)$.
There exists an index $i$ such that $1\in A_i$.
For such $i$ we have $H_{A_i}=L^2(\mathcal{A}_i)\ni 1$.
It follows that $H_{A_1}+...+H_{A_n}+span(1)=H_{A_1}+...+H_{A_n}$.
Thus $L^2(\mathcal{A}_1)+...+L^2(\mathcal{A}_n)=H_{A_1}+...+H_{A_n}$.

Let us prove (2).
Consider functions $f_k:=I_{\{k,k+1,...\}}$, $k=1,2,...$.
By the definition $f_k(t)=0$ for $t\leqslant k-1$ and
$f_k(t)=1$ for $t\geqslant k$.
If $k\in A_i$, then $f_k\in L^2(\mathcal{A}_i)$.
Since $A_1\cup...\cup A_n=\mathbb{N}$, we see that 
$f_k\in L^2(\mathcal{A}_1)+...+L^2(\mathcal{A}_n)$ for all $k\geqslant 1$.
Since $f_k-f_{k+1}=I_{\{k\}}$, we conclude that
$I_{\{k\}}\in L^2(\mathcal{A}_1)+...+L^2(\mathcal{A}_n)$ for all $k\geqslant 1$.
Thus every function of the form $a_1 I_{\{1\}}+...+a_m I_{\{m\}}$
belongs to $L^2(\mathcal{A}_1)+...+L^2(\mathcal{A}_n)$
and hence the subspace is dense in $H$.
\end{proof}

\begin{lemma}\label{L:LI}
Suppose $A_1,...,A_n$ are infinite subsets of $\mathbb{N}$ such that
$A_i\cap A_j=\emptyset$ for every pair $i\neq j$.
Let $\mathcal{A}_1,...,\mathcal{A}_n$ be the corresponding $\sigma$-algebras.
Then

(1) the subspaces $H_{A_1},...,H_{A_n}$ are linearly independent;

(2) the subspaces $L^2_0(\mathcal{A}_1),...,L^2_0(\mathcal{A}_n)$
are linearly independent.
\end{lemma}
\begin{proof}
Let us prove (1).
Assume that $x_1+...+x_n=0$, where $x_i\in H_{A_i}$, $i=1,2,...,n$.
We will show that $x_1=...=x_n=0$.
More precise, we will show by induction that 
$x_1(t)=...=x_n(t)=0$ for every $t\in\mathbb{N}$.

Let us prove this for $t=1$.
If $1\notin A_i$ for all $i=1,2,...,n$, then, by the definition of the subspace
$H_{A_i}$, we see that $x_i(1)=0$, $i=1,2,...,n$.
Now assume that $1\in A_i$ for some $i\in\{1,2,...,n\}$.
Then $1\notin A_j$ for every $j\neq i$.
It follows that $x_j(1)=0$ for every $j\neq i$.
Now the equality $x_1(1)+...+x_n(1)=0$ implies that $x_i(1)=0$.

Assume that $x_1(s)=...=x_n(s)=0$ for $s=1,2,...,t-1$.
Let us prove that $x_1(t)=...=x_n(t)=0$.
If $t\notin A_i$ for all $i=1,2,...,n$, then, by the definition of the subspace
$H_{A_i}$, we see that $x_i(t)=x_i(t-1)$ and hence $x_i(t)=0$, $i=1,2,...,n$.
Now assume that $t\in A_i$ for some $i\in\{1,2,...,n\}$.
Then $t\notin A_j$ for every $j\neq i$.
It follows that $x_j(t)=x_j(t-1)=0$ for every $j\neq i$.
Now the equality $x_1(t)+...+x_n(t)=0$ implies that $x_i(t)=0$.

Thus $x_1(t)=...=x_n(t)=0$ for every $t\in\mathbb{N}$.
This means that $x_1=...=x_n=0$ and, consequently, the subspaces
$H_{A_1},...,H_{A_n}$ are linearly independent.

Let us prove (2).
Assume that $x_1+...+x_n=0$, where $x_i\in L^2_0(\mathcal{A}_i)$, $i=1,2,...,n$.
Then $x_1(1)+...+x_n(1)=0$ and $(x_1-x_1(1))+...+(x_n-x_n(1))=0$.
But $x_i-x_i(1)\in H_{A_i}$, $i=1,2,...,n$.
Since $H_{A_1},...,H_{A_n}$ are linearly independent, we conclude that
$x_i-x_i(1)=0$ for every $i=1,2,...,n$.
Thus $x_i=x_i(1)$.
Recall that $x_i\in L^2_0(\mathcal{A}_i)$.
It follows that $E(x_i)=0$, $x_i(1)=0$, and $x_i=0$, $i=1,2,...,n$.
Hence the subspaces $L^2_0(\mathcal{A}_1),...,L^2_0(\mathcal{A}_n)$
are linearly independent.
\end{proof}

\begin{lemma}\label{L:dense}
For $k\in\mathbb{N}$ define the function $f_k:=I_{\{k,k+1,...\}}$, i.e.,
$f_k(t)=0$ for $t\leqslant k-1$ and $f_k(t)=1$ for $t\geqslant k$.
Suppose $A$ is an infinite subset of $\mathbb{N}$.
Then the linear span of the family $f_k$, $k\in A$,
is dense in $H_A$.
\end{lemma}
\begin{proof}
Let $A=\{k_1,k_2,...\}$, where $k_1<k_2<...$.
We have $f_{k_1}-f_{k_2}=I_{\{k_1,...,k_2-1\}}$,
$f_{k_2}-f_{k_3}=I_{\{k_2,...,k_3-1\}}$ and so on.
The linear span of these functions is dense in $H_A$.
It follows that the linear span of the family $f_{k_1},f_{k_2},...$
is also dense in $H_A$.
\end{proof}

\begin{lemma}\label{L:tails}
Let $\{p_k\,|\,k\geqslant 1\}$ be a sequence of positive numbers
such that the series $\sum_{k=1}^\infty p_k$ converges.
Set $r_k:=\sum_{j=k}^\infty p_j$.
The following statements are equivalent:

(1) $\sup\{r_{k+1}/r_k\,|\,k\geqslant 1\}<1$;

(2) the sequence $\{r_k/p_k\,|\,k\geqslant 1\}$ is bounded;

(3) the sequence $\{p_{k+1}/p_k\,|\,k\geqslant 1\}$ is bounded
and there exists a natural number $N$ such that
$\sup\{p_{k+N}/p_k\,|\,k\geqslant 1\}<1$.
\end{lemma}
\begin{proof}
$(1)\Leftrightarrow(2)$
This equivalence follows from the equality
$r_{k+1}/r_k=(r_k-p_k)/r_k=1-p_k/r_k=1-1/(r_k/p_k)$.

$(1)\Rightarrow (3)$
Assume that $a:=\sup\{r_{k+1}/r_k\,|\,k\geqslant 1\}<1$.
Since $(1)\Rightarrow(2)$, we conclude that 
the sequence $\{r_k/p_k\,|\,k\geqslant 1\}$ is bounded.
Set $C:=\sup\{r_k/p_k\,|\,k\geqslant 1\}$.

Since $p_{k+1}/p_k<r_k/p_k$, we conclude that the sequence
$\{p_{k+1}/p_k\,|\,k\geqslant 1\}$ is bounded.

Let us show that there exists a natural number $N$ such that
$\sup\{p_{k+N}/p_k\,|\, k\geqslant 1\}<1$.
We have $r_{k+1}/r_k\leqslant a$, $k\geqslant 1$.
It follows that $r_{k+N}/r_k\leqslant a^N$, $k\geqslant 1$, $N\geqslant 1$.
Thus 
$$
\dfrac{p_{k+N}}{p_k}<
\dfrac{r_{k+N}}{p_k}\leqslant
\dfrac{a^N r_k}{p_k}\leqslant
Ca^N
$$
for every $k\geqslant 1$.
Now choose $N$ so that $a^N<1/C$.
For such $N$ we have $\sup\{p_{k+N}/p_k\,|\,k\geqslant 1\}\leqslant Ca^N<1$. 

$(3)\Rightarrow(2)$
Let $C$ be such that $p_{k+1}/p_k\leqslant C$, $k\geqslant 1$.
Then $p_{k+1}\leqslant C p_k$, $k\geqslant 1$,
and consequently $p_{k+m}\leqslant C^m p_k$, $k\geqslant 1$.
Set $a:=\sup\{p_{k+N}/p_k\,|\,k\geqslant 1\}<1$.
Then $p_{k+N}\leqslant a p_k$, $k\geqslant 1$, and consequently
$p_{k+mN}\leqslant a^m p_k$, $k\geqslant 1$.
Consider arbitrary $m\geqslant 0$.
We can write $m=m_1 N+m_2$, where the numbers 
$m_1\geqslant 0$ and $m_2\in\{0,1,...,N-1\}$.
We have
$$
p_{k+m}=p_{k+m_1 N+m_2}\leqslant C^{m_2}p_{k+m_1 N}\leqslant
C^{m_2} a^{m_1}p_k=a^{m_1}C^{m_2}p_k
$$
and
\begin{align*}
&r_k=\sum_{j=k}^\infty p_j=\sum_{m=0}^\infty p_{k+m}\leqslant 
\sum_{m=0}^\infty a^{m_1}C^{m_2}p_k=\\
&=(1+a+a^2+...)(1+C+...+C^{N-1})p_k=
\dfrac{C^N-1}{(C-1)(1-a)}p_k
\end{align*}
for every $k\geqslant 1$.
It follows that the sequence $\{r_k/p_k\,|\,k\geqslant 1\}$ is bounded.
\end{proof}

\subsection{Proof of Theorem~\ref{T:extended N}}

Let us show that points (1), (2), (3), (4a), (4b), (5a), (5b) are pairwise equivalent.
We already know that $(1)\Leftrightarrow (2)$(see Theorem~\ref{T:IMP}).
Also, it is clear that $(4a)\Leftrightarrow(4b)$ and $(5a)\Leftrightarrow(5b)$.

$(2)\Rightarrow (4a)$
follows from Theorem~\ref{T:main}, $(1)\Leftrightarrow (2)$.

$(4a)\Rightarrow (2)$
follows from Lemma~\ref{L:LI},(2), and Theorem~\ref{T:main}, $(1)\Leftrightarrow (2)$.

$(4b)\Leftrightarrow (5b)$
follows from Lemma~\ref{L:sum L2}, (2).

$(3)\Rightarrow (4b)$
follows from Theorem~\ref{T:main}, $(1)\Leftrightarrow (2)$,
and Lemma~\ref{L:sum L2}, (1). 

$(4b)\Rightarrow (3)$
follows from Lemma~\ref{L:sum L2}, (1), Lemma~\ref{L:LI}, (1), 
and Theorem~\ref{T:main}, $(1)\Leftrightarrow (2)$.

Now we will show that $(6)\Rightarrow (3)$.
Assume (6).
By Lemma~\ref{L:tails} the sequence $\{p_{k+1}/p_k\,|\,k\geqslant 1\}$ is bounded
and there exists a natural number $N$ such that
$\sup\{p_{k+N}/p_k\,|\,k\geqslant 1\}<1$.
We will show that the system of subspaces $H_{A_1},...,H_{A_n}$ possesses the IBAP.
To this end we will use Theorem~\ref{T:IBAP and Riesz}, (2).
For each $k\geqslant 1$ define the function $f_k:=I_{\{k,k+1,...\}}$, i.e.,
$f_k(t)=0$ for $t\leqslant k-1$ and $f_k(t)=1$ for $t\geqslant k$.
If $k\in A_i$, then, clearly, $f_k\in H_{A_i}$.
For every $i\in\{1,2,...,n\}$ consider the family of elements
$(1/\sqrt{p_k})f_k$, $k\in A_i$ of the subspace $H_{A_i}$.
Lemma~\ref{L:dense} implies that the linear span of the family is dense in
$H_{A_i}$, $i=1,2,...,n$.
We claim that the family 
$\{(1/\sqrt{p_k})f_k\,|\,k\in A_1\cup A_2\cup...\cup A_n=\mathbb{N}\}$
is a Riesz basis for $H$.
If we will prove the claim, then by Theorem~\ref{T:IBAP and Riesz}, (2),
the system of subspaces $H_{A_1},...,H_{A_n}$ possesses the IBAP.

So, let us prove that the family $\{(1/\sqrt{p_k})f_k\,|\,k\geqslant 1\}$
is a Riesz basis for $H$.
Set $e_k:=(1/\sqrt{p_k})I_{\{k\}}$, $k\geqslant 1$.
Then the family $\{e_k\,|\,k\geqslant 1\}$ is an orthonormal basis for $H$.
Define the continuous linear operator (weighted shift) $S:H\to H$
by $Se_k:=\sqrt{p_{k+1}/p_k}e_{k+1}$, $k\geqslant 1$, that is,
$$
S\left(\sum_{k=1}^\infty a_k e_k\right):=
\sum_{k=1}^\infty a_k\sqrt{p_{k+1}/p_k} e_{k+1},\quad
\sum_{k=1}^\infty|a_k|^2<\infty.
$$
Since the sequence $\{p_{k+1}/p_k\,|\,k\geqslant 1\}$ is bounded, we see
that the operator $S$ is well-defined and is a continuous linear operator.
We have $S(\sqrt{p_k}e_k)=\sqrt{p_{k+1}}e_{k+1}$.
It follows that $S^N(\sqrt{p_k}e_k)=\sqrt{p_{k+N}}e_{k+N}$, i.e.,
$S^N e_k=\sqrt{p_{k+N}/p_k}e_{k+N}$.
Thus
$$
S^N\left(\sum_{k=1}^\infty a_k e_k\right)=
\sum_{k=1}^\infty a_k\sqrt{p_{k+N}/p_k}e_{k+N},\quad
\sum_{k=1}^\infty|a_k|^2<\infty.
$$
It follows that 
$\|S^N\|=\sup\{\sqrt{p_{k+N}/p_k}\,|\,k\geqslant 1\}<1$.
Consequently, the operator $I-S^N$ is an isomorphism$\Rightarrow$
the operator $I-S$ is also an isomorphism.
We have
$$
Sf_k=
S\left(\sum_{j=k}^\infty I_{\{j\}}\right)=
S\left(\sum_{j=k}^\infty\sqrt{p_j}e_j\right)=
\sum_{j=k}^\infty\sqrt{p_{j+1}}e_{j+1}=
f_{k+1}.
$$
Since $f_k-f_{k+1}=I_{\{k\}}=\sqrt{p_k}e_k$, we conclude that
$(I-S)f_k=\sqrt{p_k}e_k$,
$(I-S)((1/\sqrt{p_k})f_k)=e_k$,
$(1/\sqrt{p_k})f_k=(I-S)^{-1}e_k$, $k\geqslant 1$.
Therefore the family $\{(1/\sqrt{p_k})f_k\,|\,k\geqslant 1\}$
is a Riesz basis for $H$.

Finally, assume that there exists a number $k_0\geqslant 1$ 
such that for every $k\geqslant k_0$
the numbers $k$ and $k+1$ belong to distinct sets $A_i$.
Let us show that $(3)\Rightarrow(6)$.
Assume that the system of subspaces $H_{A_1},...,H_{A_n}$ possesses the IBAP.
Then for every pair of distinct indices $i,j$ the pair of subspaces
$H_{A_i}$ and $H_{A_j}$ possesses the IBAP.
Consequently, $\delta(H_{A_i},H_{A_j})>0$.
Set $\delta:=\min\{\delta(H_{A_i},H_{A_j})\,|\,i\neq j\}>0$.
Then for arbitrary $i\neq j$ and elements $x\in H_{A_i}$, $y\in H_{A_j}$ we have
$$
\|x-y\|\geqslant dist(x,H_{A_j})\geqslant\delta(H_{A_i},H_{A_j})\|x\|
\geqslant\delta\|x\|.
$$
Thus
\begin{equation}\label{ineq:x,y}
\|x-y\|\geqslant\delta\|x\|.
\end{equation}
For each $k\geqslant 1$ define the function $f_k:=I_{\{k,k+1,...\}}$, i.e.,
$f_k(t)=0$ for $t\leqslant k-1$ and $f_k(t)=1$ for $t\geqslant k$.
If $k\in A_i$, then, clearly, $f_k\in H_{A_i}$.
Since for every $k\geqslant k_0$ the numbers $k,k+1$ belong to distinct sets $A_i$,
inequality~\eqref{ineq:x,y} implies that
\begin{equation}\label{ineq:f_k and f_k+1}
\|f_k-f_{k+1}\|\geqslant\delta\|f_k\|,\quad k\geqslant k_0.
\end{equation}
Since $f_k-f_{k+1}=I_{\{k\}}$, we see that $\|f_k-f_{k+1}\|=\sqrt{p_k}$.
Also we have $\|f_k\|=\sqrt{\sum_{j=k}^\infty p_j}=\sqrt{r_k}$.
Now inequality \eqref{ineq:f_k and f_k+1} implies that
$\sqrt{p_k}\geqslant\delta\sqrt{r_k}$, $k\geqslant k_0$, that is,
$r_k/p_k\leqslant 1/\delta^2$, $k\geqslant k_0$.
Thus the sequence $\{r_k/p_k\,|\,k\geqslant 1\}$ is bounded.
By Lemma~\ref{L:tails} we get (6).

The proof is complete.

\subsection{Example 2 ($\Omega=[a,b)$)}\label{SS:Omega=[a,b)}

Consider the space $\Omega:=[a,b)$, where $a,b\in\mathbb{R}$, $a<b$,
with the Borel $\sigma$-algebra $\mathcal{F}=\mathcal{B}([a,b))$.
Let $\mu$ be a probability measure on $\mathcal{F}$.
Set $H:=L^2([a,b),\mathcal{F},\mu)$.
For a sequence of points $\pi=\{a_1,a_2,a_3,...\}$,
where $a<a_1<a_2<a_3<...$ and $a_k\to b$ as $k\to\infty$,
define the partition of $[a,b)$, $part(\pi)$, by
$$
part(\pi):=\{[a,a_1),[a_1,a_2),[a_2,a_3),[a_3,a_4),...\}.
$$
Let $\sigma a(\pi)$ be the $\sigma$-algebra generated by $part(\pi)$.

We will study the following problem.
Let $\pi_1,...,\pi_n$ be sequences of points.
Question: when the collection of $\sigma$-algebras
$\sigma a(\pi_1),...,\sigma a(\pi_n)$ possesses the IMP?
Equivalently, when the system of subspaces
$L^2_0(\sigma a(\pi_1)),...,L^2_0(\sigma a(\pi_n))$ possesses the IBAP?
In the following theorem we establish relations between 
the IMP, the IBAP, closedness of the sum of marginal subspaces and
``fast decreasing'' of tails of the measure $\mu$.

\begin{theorem}\label{T:[a,b)}
Let $\pi_i=\{a^{(i)}_1,a^{(i)}_2,...\}$ be a sequence of points such that
$a<a^{(i)}_1<a^{(i)}_2<...$ and $a^{(i)}_k\to b$ as $k\to\infty$
for every $i=1,2,...,n$.
Assume that $\pi_i\cap \pi_j=\emptyset$, $i\neq j$.
Let $\pi_1\cup...\cup\pi_n=\{b_2,b_3,b_4,...\}$, where
$a=:b_1<b_2<b_3<...$.
Assume that $\mu([b_k,b_{k+1}))>0$ for $k\geqslant 1$.
Consider the following statements:

(1) the collection $\sigma a(\pi_1),...,\sigma a(\pi_n)$ possesses the IMP;

(2) the system of subspaces
$L^2_0(\sigma a(\pi_1)),...,L^2_0(\sigma a(\pi_n))$ possesses the IBAP;

(3) the subspace
$L^2_0(\sigma a(\pi_1))+...+L^2_0(\sigma a(\pi_n))$ 
is closed in $L^2(\mathcal{B}([a,b)))$;

(4)$L^2_0(\sigma a(\pi_1))+...+L^2_0(\sigma a(\pi_n))=
L^2_0(\sigma a(\pi_1\cup...\cup\pi_n))$;

(5) $\sup\{\mu([b_{k+1},b))/\mu([b_k,b))\,|\,k\geqslant 1\}<1$.

Then statements (1), (2), (3), (4) are pairwise equivalent and
$(5)\Rightarrow (1)$.
If there exists $k_0\geqslant 1$ such that
for arbitrary $i\in\{1,2,...,n\}$ and arbitrary $k\geqslant k_0$
the interval $(a^{(i)}_k,a^{(i)}_{k+1})$ contains a point $a^{(j)}_l$,
then $(1)\Rightarrow(5)$.
\end{theorem}

\begin{remark}
The condition $\pi_i\cap\pi_j=\emptyset$, $i\neq j$, is necessary for
the collection $\sigma a(\pi_1),...,\sigma a(\pi_n)$ to possess the IMP.
Indeed, assume that $\pi_i\cap\pi_j\neq\emptyset$ for some $i\neq j$.
Let $c\in\pi_i\cap\pi_j$, then $c\in\pi_i$ and $c\in\pi_j$.
Consider the function $f=\mu([c,b))I_{[a,c)}-\mu([a,c))I_{[c,b)}$.
It is clear that $f\in L^2(\sigma a(\pi_i))$, $f\in L^2(\sigma a(\pi_j))$, and $E(f)=0$.
Thus $f\in L^2_0(\sigma a(\pi_i))\cap L^2_0(\sigma a(\pi_j))$.
Consequently, 
$L^2_0(\sigma a(\pi_i))\cap L^2_0(\sigma a(\pi_j))\neq\{0\}$
$\Rightarrow$
the pair $L^2_0(\sigma a(\pi_i)),L^2_0(\sigma a(\pi_j))$ does not possess the IBAP
$\Rightarrow$
the system $L^2_0(\sigma a(\pi_1)),...,L^2_0(\sigma a(\pi_n))$ does not possess the IBAP
$\Rightarrow$
the collection $\sigma a(\pi_1),...,\sigma a(\pi_n)$ does not possess the IMP.
\end{remark}

\begin{proof}[Proof of Theorem~\ref{T:[a,b)}]
Theorem~\ref{T:[a,b)} follows from Theorem~\ref{T:N}.
Let us show this.

Define the measure $\nu$ on $2^{\mathbb{N}}$ by
$\nu(\{k\}):=\mu([b_k,b_{k+1}))$, $k\geqslant 1$.
It is clear that $\nu$ is a probability measure and 
$\nu(\{k\})>0$, $k\geqslant 1$.
Define the operator 
$U:L^2(\mathbb{N},2^{\mathbb{N}},\nu)\to L^2(\sigma a(\pi_1\cup...\cup\pi_n))$ by
$$
(Ux)|_{[b_k,b_{k+1})}:=x(k),\quad k\geqslant 1,\quad
x\in L^2(\mathbb{N},2^{\mathbb{N}},\nu).
$$
One can easily check that $U$ is well-defined, unitary, and
$E(Ux)=E(x), x\in L^2(\mathbb{N},2^{\mathbb{N}},\nu)$. 

Define sets $A_i\subset\mathbb{N}\setminus\{1\}$, $i=1,2,...,n$ by
$A_i:=\{k\,|\,b_k\in\pi_i\}$, $i=1,2,...,n$.
It is clear that each set $A_i$ is infinite and the sets
$A_1,...,A_n$ form a partition of $\mathbb{N}\setminus\{1\}$.
It is easily seen that
$U(L^2(\sigma a(A_i)))=L^2(\sigma a(\pi_i))$, $i=1,2,...,n$,
$U(L^2_0(2^{\mathbb{N}}))=L^2_0(\sigma a(\pi_1\cup...\cup\pi_n))$, and
$U(L^2_0(\sigma a(A_i)))=L^2_0(\sigma a(\pi_i))$, $i=1,2,...,n$.
Thus

(i) 
the system $L^2_0(\sigma a(\pi_1)),...,L^2_0(\sigma a(\pi_n))$
possesses the IBAP in $L^2(\mathcal{B}([a,b)))$
$\Leftrightarrow$
the system $L^2_0(\sigma a(\pi_1)),...,L^2_0(\sigma a(\pi_n))$
possesses the IBAP in $L^2(\sigma a(\pi_1\cup...\cup\pi_n))$
$\Leftrightarrow$
the system $L^2_0(\sigma a(A_1)),...,L^2_0(\sigma a(A_n))$
possesses the IBAP in $L^2(2^{\mathbb{N}})$;

(ii)
the subspace $L^2_0(\sigma a(\pi_1))+...+L^2_0(\sigma a(\pi_n))$ 
is closed in $L^2(\mathcal{B}([a,b)))$
$\Leftrightarrow$
the subspace $L^2_0(\sigma a(\pi_1))+...+L^2_0(\sigma a(\pi_n))$
is closed in $L^2(\sigma a(\pi_1\cup...\cup\pi_n))$
$\Leftrightarrow$
the subspace $L^2_0(\sigma a(A_1))+...+L^2_0(\sigma a(A_n))$
is closed in $L^2(2^{\mathbb{N}})$;

(iii)
$L^2_0(\sigma a(\pi_1))+...+L^2_0(\sigma a(\pi_n))=
L^2_0(\sigma a(\pi_1\cup...\cup\pi_n))$
$\Leftrightarrow$
$L^2_0(\sigma a(A_1))+...+L^2_0(\sigma a(A_n))=L^2_0(2^{\mathbb{N}})$.

Let us show that points (1), (2), (3), (4) of Theorem~\ref{T:[a,b)}
are pairwise equivalent.
We already know that $(1)\Leftrightarrow(2)$ (see Theorem~\ref{T:IMP}).
Theorem~\ref{T:N} and (i), (ii), (iii) imply that
$(2)\Leftrightarrow(3)\Leftrightarrow(4)$.
Therefore (1), (2), (3), (4) are pairwise equivalent.

Let us show that $(5)\Rightarrow(1)$.
Assume (5).
Since $r_k=\sum_{j=k}^\infty\nu(\{j\})=
\sum_{j=k}^\infty\mu([b_j,b_{j+1}))=\mu([b_k,b))$,
we see that (5) can be rewritten as
$\sup\{r_{k+1}/r_k\,|\,k\geqslant 1\}<1$.
By Theorem~\ref{T:N} 
the collection $\sigma a(A_1),...,\sigma a(A_n)$ possesses the IMP
$\Rightarrow$
the system $L^2_0(\sigma a(A_1)),...,L^2_0(\sigma a(A_n))$ possesses the IBAP
$\Rightarrow$
the system $L^2_0(\sigma a(\pi_1)),...,L^2_0(\sigma a(\pi_n))$ possesses the IBAP
$\Rightarrow$
the collection $\sigma a(\pi_1),...,\sigma a(\pi_n)$ possesses the IMP.

Finally, assume that there exists $k_0\geqslant 1$ such that
for arbitrary $i\in\{1,2,...,n\}$ and arbitrary $k\geqslant k_0$
the interval $(a^{(i)}_k,a^{(i)}_{k+1})$ contains a point $a^{(j)}_l$.
We have to show that $(1)\Rightarrow(5)$.
Suppose that
the collection $\sigma a(\pi_1),...,\sigma a(\pi_n)$ possesses the IMP
$\Rightarrow$
the system $L^2_0(\sigma a(\pi_1)),...,L^2_0(\sigma a(\pi_n))$ possesses the IBAP
$\Rightarrow$
the system $L^2_0(\sigma a(A_1)),...,L^2_0(\sigma a(A_n))$ possesses the IBAP
$\Rightarrow$
the collection $\sigma a(A_1),...,\sigma a(A_n)$ possesses the IMP.
Clearly, there exists a number $\widetilde{k}_0\geqslant 2$
such that for arbitrary $k\geqslant\widetilde{k}_0$ the numbers
$k$ and $k+1$ belong to distinct sets $A_i$.
Now Theorem~\ref{T:N} implies that
$\sup\{r_{k+1}/r_k\,|\,k\geqslant 1\}<1$
$\Rightarrow(5)$.

The proof is complete.
\end{proof}

\subsection{A sufficient condition for a collection of
sub-$\sigma$-algebras to possess the IMP}

We will give a sufficient condition for a collection of
sub-$\sigma$-algebras to possess the IMP.
Our result is motivated by the following result of
Peter J.~Bickel, Ya'akov Ritov and Jon A.~Wellner for a pair of
sub-$\sigma$-algebras.

Let $(\Omega_1,\mathcal{A},\mu_1)$ and $(\Omega_2,\mathcal{B},\mu_2)$ 
be two probability spaces.
Set $(\Omega,\mathcal{F}):=(\Omega_1\times\Omega_2,\mathcal{A}\otimes\mathcal{B})$.
Suppose $\mu$ is a probability measure on $\mathcal{A}\otimes\mathcal{B}$ 
with marginals $\mu_1$ and $\mu_2$
(that is, $\mu(A\times\Omega_2)=\mu_1(A), A\in\mathcal{A}$ and 
$\mu(\Omega_1\times B)=\mu_2(B), B\in\mathcal{B}$).
Let $\mathcal{F}_1=\mathcal{A}\times\Omega_2=\{A\times\Omega_2\mid A\in\mathcal{A}\}$ and
$\mathcal{F}_2=\Omega_1\times\mathcal{B}=\{\Omega_1\times B\mid B\in\mathcal{B}\}$.
Then $L^2(\mathcal{F}_1)$ consists of (equivalence classes of) 
random variables of the form $\xi(\omega_1,\omega_2)=f(\omega_1)$ with
$f\in L^2(\Omega_1,\mathcal{A},\mu_1)$.
Similarly, $L^2(\mathcal{F}_2)$ consists of (equivalence classes of) 
random variables of the form $\xi(\omega_1,\omega_2)=g(\omega_2)$ with
$g\in L^2(\Omega_2,\mathcal{B},\mu_2)$.
Bickel, Ritov and Wellner showed that if there exists $\alpha>0$ such that
$$
\mu(A\times B)\geqslant\alpha\mu_1(A)\mu_2(B)
$$
for arbitrary $A\in\mathcal{A}, B\in\mathcal{B}$,
then the subspaces $L^2_0(\mathcal{F}_1)$ and $L^2_0(\mathcal{F}_2)$ 
are linearly independent (i.e., their intersection is $\{0\}$)
and their sum is closed in $L^2(\mathcal{F})$.
This means that the pair of subspaces 
$L^2_0(\mathcal{F}_1)$, $L^2_0(\mathcal{F}_2)$ possesses the IBAP.
Consequently, the pair $\mathcal{F}_1,\mathcal{F}_2$ possesses the IMP.

We will prove the following theorem, which is a generalization of
the result of Bickel, Ritov and Wellner.

\begin{theorem}\label{T:sufficient for IMP}
Let $(\Omega,\mathcal{F},\mu)$ be a probability space and
$\mathcal{F}_1,...,\mathcal{F}_n$ be sub-$\sigma$-algebras of $\mathcal{F}$.
If there exists a number $\alpha>0$ such that
$$
\mu(A_1\cap...\cap A_n)\geqslant\alpha\mu(A_1)...\mu(A_n)
$$
for arbitrary $A_1\in\mathcal{F}_1,...,A_n\in\mathcal{F}_n$, then
the collection of sub-$\sigma$-algebras
$\mathcal{F}_1,...,\mathcal{F}_n$ possesses the IMP.
\end{theorem}
\begin{proof}
We will show that the system of subspaces
$L^2_0(\mathcal{F}_1),...,L^2_0(\mathcal{F}_n)$ 
possesses the IBAP in $L^2(\mathcal{F})$.

\textbf{1.}
Consider the probability space $(\Omega,\mathcal{F},\mu)$, 
the measurable space $(\Omega,\mathcal{F}_1)\otimes...\otimes (\Omega,\mathcal{F}_n)=
(\Omega^n,\mathcal{F}_1\otimes...\otimes\mathcal{F}_n)$ and
a mapping $T:\Omega\to\Omega^n$ defined by
$T\omega=(\omega,...,\omega)$, $\omega\in\Omega$.
Since $T^{-1}(A_1\times...\times A_n)=A_1\cap...\cap A_n$ for
$A_1\in\mathcal{F}_1,...,A_n\in\mathcal{F}_n$, we see that $T$ is measurable.
Thus we can define the pushforward measure $\nu=T_{\ast}\mu$ 
on $\mathcal{F}_1\otimes...\otimes\mathcal{F}_n$ by
$\nu(A):=\mu(T^{-1}(A))$, $A\in\mathcal{F}_1\otimes...\otimes\mathcal{F}_n$.
The measure $\nu$ has the following properties.

Firstly, we have the change-in-variables formula: 
if a function $f:\Omega^n\to\mathbb{K}$ is 
$\mathcal{F}_1\otimes...\otimes\mathcal{F}_n$-measurable, then
\begin{equation}\label{eq:change-in-variables}
\int_{\Omega^n}f(\omega_1,...,\omega_n)d\nu=
\int_{\Omega}f(\omega,...,\omega)d\mu
\end{equation}
(the equality means that the first integral exists 
if and only if the second exists, and if they exist, then they are equal).

To formulate the second property of $\nu$ denote by $\mu_i$ 
the restriction of $\mu$ to $\mathcal{F}_i$, $i=1,2,....n$.
For arbitrary $A_1\in\mathcal{F}_1,...,A_n\in\mathcal{F}_n$ we have
\begin{align*}
\nu(A_1\times...\times A_n)&=\mu(A_1\cap...\cap A_n)\geqslant
\alpha\mu(A_1)...\mu(A_n)=\\
&=\alpha\mu_1(A_1)...\mu_n(A_n)=
\alpha(\mu_1\otimes...\otimes\mu_n)(A_1\times...\times A_n).
\end{align*}
It follows that 
\begin{equation}\label{ineq:greater}
\nu(A)\geqslant\alpha(\mu_1\otimes...\otimes\mu_n)(A)
\end{equation}
for each set $A\in\mathcal{F}_1\otimes...\otimes\mathcal{F}_n$.

Since $(\mu_1\otimes...\otimes\mu_n)(A)\leqslant (1/\alpha)\nu(A)$,
$A\in\mathcal{F}_1\otimes...\otimes\mathcal{F}_n$, we see that
$\mu_1\otimes...\otimes\mu_n$ is absolutely continuous with respect to $\nu$.
By the Radon-Nikodym theorem there exists an 
$\mathcal{F}_1\otimes...\otimes\mathcal{F}_n$-measurable function
$h:\Omega^n\to\mathbb{R}$ such that
$$
(\mu_1\otimes...\otimes\mu_n)(A)=\int_A h(\omega_1,...,\omega_n)d\nu
$$
for every set $A\in\mathcal{F}_1\otimes...\otimes\mathcal{F}_n$.
Since $0\leqslant(\mu_1\otimes...\otimes\mu_n)(A)\leqslant(1/\alpha)\nu(A)$,
$A\in\mathcal{F}_1\otimes...\otimes\mathcal{F}_n$, we conclude that
$0\leqslant h(\omega_1,...,\omega_n)\leqslant 1/\alpha$ for $\nu$-almost every
point $(\omega_1,...,\omega_n)\in\Omega^n$.
We can and will assume that $0\leqslant h(\omega_1,...,\omega_n)\leqslant 1/\alpha$
for every point $(\omega_1,...,\omega_n)\in\Omega^n$.

Also, if a function $f:\Omega^n\to\mathbb{K}$ is 
$\mathcal{F}_1\otimes...\otimes\mathcal{F}_n$-measurable, then
\begin{equation}\label{eq:Radon-Nikodym}
\int_{\Omega^n}f(\omega_1,...,\omega_n)d(\mu_1\otimes...\otimes\mu_n)=
\int_{\Omega^n}f(\omega_1,...,\omega_n)h(\omega_1,...,\omega_n)d\nu
\end{equation}
(the equality means that the first integral exists 
if and only if the second exists, and if they exist, then they are equal).

\textbf{2.}
In what follows we will often use the following simple fact.
If $\mathcal{A}$ is a sub-$\sigma$-algebra of $\mathcal{F}$
and a random variable $\xi:\Omega\to\mathbb{K}$ is $\mathcal{A}$-measurable, then
$\int_\Omega\xi d\mu=\int_\Omega\xi d(\mu|_{\mathcal{A}})$
(the equality means that the first integral exists if and only if 
the second exists, and if they exist, then they are equal).

\textbf{3.}
Now we are ready to prove that the system of subspaces
$L^2_0(\mathcal{F}_1),...,L^2_0(\mathcal{F}_n)$ possesses the IBAP.
Consider arbitrary elements 
$\xi_1\in L^2_0(\mathcal{F}_1),...,\xi_n\in L^2_0(\mathcal{F}_n)$.
In what follows we assume that $\xi_k$ is an $\mathcal{F}_k$-measurable
random variable, $k=1,2,...,n$.
Define a random variable $\xi:\Omega\to\mathbb{K}$ by
$$
\xi(\omega):=(\xi_1(\omega)+...+\xi_n(\omega))h(\omega,...,\omega),
\quad\omega\in\Omega.
$$
Since $0\leqslant h(\omega,...,\omega)\leqslant 1/\alpha$ for every $\omega\in\Omega$,
we conclude that $\xi\in L^2(\mathcal{F})$.
Denote by $P_k$ the orthogonal projection onto $L^2_0(\mathcal{F}_k)$ in
$L^2(\mathcal{F})$, $k=1,2,...,n$.
Let us show that $P_1\xi=\xi_1$.
To this end we will prove that $\xi-\xi_1\bot L^2_0(\mathcal{F}_1)$.
Consider arbitrary element $\eta_1\in L^2_0(\mathcal{F}_1)$.
In what follows we assume that $\eta_1$ is an $\mathcal{F}_1$-measurable
random variable.
Using~\eqref{eq:change-in-variables},~\eqref{eq:Radon-Nikodym} and
Fubini theorem, we get
\begin{align*}
&\langle\xi,\eta_1\rangle=
\int_\Omega\xi(\omega)\overline{\eta_1(\omega)}d\mu=
\int_\Omega(\xi_1(\omega)+...+\xi_n(\omega))h(\omega,...,\omega)
\overline{\eta_1(\omega)}d\mu=\\
&=\int_{\Omega^n}(\xi_1(\omega_1)+...+\xi_n(\omega_n))
\overline{\eta_1(\omega_1)}h(\omega_1,...,\omega_n)d\nu=\\
&=\int_{\Omega^n}(\xi_1(\omega_1)+...+\xi_n(\omega_n))
\overline{\eta_1(\omega_1)}d(\mu_1\otimes...\otimes\mu_n)=\\
&=\int_{\Omega^n}\xi_1(\omega_1)\overline{\eta_1(\omega_1)}
d(\mu_1\otimes...\otimes\mu_n)+
\sum_{k=2}^n\int_{\Omega^n}\xi_k(\omega_k)\overline{\eta_1(\omega_1)}
d(\mu_1\otimes...\otimes\mu_n)=\\
&=\int_\Omega\xi_1(\omega_1)\overline{\eta_1(\omega_1)}d\mu_1+
\sum_{k=2}^n\int_\Omega\overline{\eta_1(\omega_1)}d\mu_1\cdot
\int_\Omega\xi_k(\omega_k)d\mu_k=\\
&=E(\xi_1\overline{\eta_1})+\sum_{k=2}^n E(\overline{\eta_1})E(\xi_k)=
\langle\xi_1,\eta_1\rangle.
\end{align*}
It follows that $\langle\xi-\xi_1,\eta_1\rangle=0$.
Since $\eta_1\in L^2_0(\mathcal{F}_1)$ is arbitrary, we conclude that
$\xi-\xi_1\bot L^2_0(\mathcal{F}_1)$ and consequently $P_1\xi=\xi_1$.
Similarly, one can show that $P_k\xi=\xi_k$ for $k=2,...,n$.
Thus the system of subspaces $L^2_0(\mathcal{F}_1),...,L^2_0(\mathcal{F}_n)$
possesses the IBAP. 
\end{proof}

\begin{remark}
One can prove Theorem~\ref{T:sufficient for IMP}
by using Theorem~\ref{T:main}, $(1)\Leftrightarrow(3)$.
Indeed, for arbitrary elements
$\xi_1\in L^2_0(\mathcal{F}_1),...,\xi_n\in L^2_0(\mathcal{F}_n)$
by using~\eqref{eq:change-in-variables},~\eqref{ineq:greater} and
Fubini theorem we have
\begin{align*}
&\|\xi_1+...+\xi_n\|^2=
\int_\Omega|\xi_1(\omega)+...+\xi_n(\omega)|^2 d\mu=
\int_{\Omega^n}|\xi_1(\omega_1)+...+\xi_n(\omega_n)|^2 d\nu\geqslant\\
&\geqslant\alpha\int_{\Omega^n}|\xi_1(\omega_1)+...+\xi_n(\omega_n)|^2
d(\mu_1\otimes...\otimes\mu_n)=\\
&=\alpha\sum_{k=1}^n\int_{\Omega^n}|\xi_k(\omega_k)|^2
d(\mu_1\otimes...\otimes\mu_n)+
\alpha\sum_{k\neq l}\int_{\Omega^n}\xi_k(\omega_k)\overline{\xi_l(\omega_l)}
d(\mu_1\otimes...\otimes\mu_n)=\\
&=\alpha\sum_{k=1}^n\int_\Omega|\xi_k(\omega_k)|^2 d\mu_k+
\alpha\sum_{k\neq l}\int_\Omega \xi_k(\omega_k)d\mu_k\cdot
\int_\Omega \overline{\xi_l(\omega_l)}d\mu_l=\\
&=\alpha\sum_{k=1}^n\|\xi_k\|^2+\alpha\sum_{k\neq l}E(\xi_k)E(\overline{\xi_l})=
\alpha\sum_{k=1}^n\|\xi_k\|^2.
\end{align*}
Therefore $\|\xi_1+...+\xi_n\|\geqslant
\sqrt{\alpha}\sqrt{\|\xi_1\|^2+...+\|\xi_n\|^2}$.
By Theorem~\ref{T:main}, $(1)\Leftrightarrow(3)$, the system of subspaces
$L^2_0(\mathcal{F}_1),...,L^2_0(\mathcal{F}_n)$ possesses the IBAP.
\end{remark}

\subsection*{Acknowledgements}

The research was funded by 
Institute of Mathematics of NAS of Ukraine.
A part of this research was supported by 
the Project 2017-3M from the Department of Targeted Training of
Taras Shevchenko National University of Kyiv at the NAS of Ukraine.

\end{document}